\documentclass[11pt]{amsart}
\usepackage{amssymb,latexsym, amsmath, enumerate, amsthm, mathtools}
\usepackage{color}
\usepackage{graphicx,overpic}
\usepackage{hyperref}
\usepackage{comment}
\usepackage[title,titletoc]{appendix}
\hypersetup{
      CJKbookmarks=true
	colorlinks=true,
	linkcolor=blue,
	filecolor=blue,      
	urlcolor=blue,
	citecolor=blue,
}

\usepackage[margin=1.5in]{geometry}

\usepackage{setspace}
\def\R{\mathbb R}
\def\N{\mathbb N}

\def\G{\mathbb G}
\def\H{\mathbb H}
\def\E{\mathbb E}
\def\X{\mathrm X}
\def\Y{\mathrm Y}
\def\Z{\mathrm Z}
\def\T{\mathrm T}
\def\t{\tilde}
\def\b{\bar}
\def\wt{\widetilde}
\def\cR{\mathcal{R}}

\def\U{\Upsilon}
\def\HH{\mathcal{H}}

\def\g{\mathfrak{g}}
\def\D{\mathcal{D}}
\def\I{\mathrm{Id}}
\def\U{\mathcal{U}}
\def\bS{\mathbf{S}}
\newcommand{\cX}{\mathcal{X}}

\newtheorem{theorem}{Theorem}
\newtheorem{lemma}[theorem]{Lemma}
\newtheorem{corollary}[theorem]{Corollary}
\newtheorem{proposition}[theorem]{Proposition}

\theoremstyle{definition}
\newtheorem{definition}[theorem]{Definition}

\theoremstyle{remark}
\newtheorem{remark}[theorem]{Remark}
\newtheorem{example}[theorem]{Example}

\numberwithin{theorem}{section}

\title[H-semiconcavity for square of CC distance]{Horizontal semiconcavity for the square of Carnot-Carath\'eodory distance on step 2 Carnot groups
and applications to Hamilton-Jacobi equations}
\date{\today}

\author[F. Dragoni]{Federica Dragoni}
\address[F.~Dragoni]{\noindent School of Mathematics, Cardiff University, Cardiff,  UK, \tt dragonif@cardiff.ac.uk.} 

\author[Q. Liu]{Qing Liu}
\address[Q.~Liu]{Geometric Partial Differential Equations Unit, Okinawa Institute of Science and Technology Graduate University, Okinawa 904-0495, Japan, {\tt qing.liu@oist.jp}}

\author[Y. Zhang]{Ye Zhang}
\address[Y.~Zhang]{Analysis on Metric Spaces Unit, Okinawa Institute of Science and Technology Graduate University, Okinawa 904-0495, Japan, {\tt ye.zhang2@oist.jp}}

\begin{document}

\begin{abstract}
We show that the square of Carnot-Carath\'eodory distance from the origin, in step 2 Carnot groups, enjoys the horizontal semiconcavity (h-semiconcavity) everywhere in the group including the origin. We first give a proof in the case of ideal Carnot groups, based on the simple group structure as well as estimates for the Euclidean semiconcavity. Our proof of the general result involves more geometric properties of step 2 Carnot groups.  We further apply our h-semiconcavity result to show h-semiconcavity of the viscosity solutions to a class of non-coercive evolutive Hamilton-Jacobi equations by using the Hopf-Lax formula associated to the Carnot-Carath\'eodory metric. 
\end{abstract}

\subjclass[2020]{35R03, 35D40, 49L25}
\keywords{Carnot-Carath\'eodry distance, Carnot groups, Heisenberg group,  horizontal semiconcavity, Hamilton-Jacobi equations, Hopf-Lax formula}

\maketitle


\section{Introduction.}
\setcounter{equation}{0}

Semiconvexity and semiconcavity properties are key regularity properties for functions, related to bounds for the second derivatives, and applied in many different contexts. We refer to the monograph by Cannarsa and Sinestrari \cite{CS04} for an overview on the topic. One of the most interesting functions, where one can apply this property, is the distance function. It is easily seen that the standard Euclidean distance is convex and thus semiconvex everywhere but not semiconcave. However, the squared Euclidean distance to a given point is both semiconvex and semiconcave, since its Hessian is a constant nonnegative matrix. This leads to many interesting consequences, and it is somehow behind the successful use of the squared distance to prove many results in PDEs. This  opens to question the relation between semiconvexity/semiconcavity and distance functions in different geometrical settings.

In the case of Carnot groups,  there is a vast literature investigating the notion of convexity (or concavity) associated to their sub-Riemannian structure. Later in the paper we will review several known notions that can be defined in these spaces. Let us just quickly recall the notion of horizontal convexity (h-convexity for short) introduced by Lu-Mandredi-Stroffolini in the Heisenberg group \cite{LMS03}; see also \cite{JLMS07} for extension to more general Carnot groups. At the same time, the notion of h-convexity is also studied independently by Danielli-Garofalo-Nhieu in \cite{DGN03} by adapting the standard convexity definition to the algebraic structure of Carnot groups. The notion of h-concavity can be symmetrically defined; see Definition \ref{h-concavity} for a precise definition. Such a notion was later generalized by Bardi and the first author in \cite{BD11}  with a more geometrical approach that does not require any underlying Lie group structure and cover sub-Riemannian structures up to the case of Carnot-type H\"ormander vector fields. 

While many results are known for h-convex or h-concave functions, less research has been conducted on their semiconvex/semiconcave counterparts, which can be defined by easily adapting the concept of h-convexity/h-concavity. As an analogue of the Euclidean case, in \cite{BD11} it is shown that these notions are equivalent to bounds in the viscosity sense for the intrinsic Hessian. 
This equivalence demonstrates that h-semiconvexity/h-semiconcavity serves as a natural sub-Riemannian generalization, elucidating why proving h-semiconvexity/h-semiconcavity properties is exceptionally useful in studying degenerate PDEs associated with Carnot groups.

In this paper we are interested in exploring such semiconcavity for functions related to the metric of a class of sub-Riemannian manifolds. 
In the setting of Carnot groups, various notions of metric can be considered and they turn out to be all locally equivalent. In this work, we discuss the Carnot-Carath\'eodory distance (CC distance for short). 
It is also the geodesic distance, which can be defined as the minimal length among admissible curves joining two given points. A precise definition is given in Definition \ref{C-C distance}. It has many important analytic and geometric properties; for example it is the only distance solving the eikonal equation in this geometrical setting, as shown in \cite{{D07},{MS02}}. 

We attempt to provide h-concavity results for the CC distance in Carnot groups. Our main result of this paper is the following theorem.
\begin{theorem} \label{t2}
Let $\G$ be a step 2 Carnot group with CC distance $d$. Then $d^2(\cdot, 0)$ is h-semiconcave in $\G$.
\end{theorem}

Although our result holds for general Carnot groups of step 2, we shall first focus on the so-called ideal Carnot groups, a subclass of step 2 Carnot groups that includes the Heisenberg group and general H-type groups as simple yet significant examples. We choose to begin with a proof for this special case, because it only makes use of several well-understood results concerning the group structure and the local Euclidean semiconcavity without requiring any further techniques. An important property of ideal Carnot groups is that the abnormal set, which consists of endpoints of abnormal minimizing geodesics starting from the group identity, contains only the identity itself. Under certain assumptions for a general Carnot group, it is shown in \cite{CR08} that the CC distance from the identity is locally Euclidean semiconcave, and therefore locally h-semiconcave, away from the abnormal set. See also \cite[Theorem 5.9]{FR10} for this result.
 However, despite many potential applications in nonlinear analysis and PDEs, the behavior near the identity has not been well understood in the literature. 

Our analysis complements \cite{CR08, FR10}, extending the h-semiconcavity of $d^2(\cdot, 0)$ to the identity, and thus obtaining the regularity globally. In the case of ideal Carnot groups, we prove the h-semiconcavity of $d^2(\cdot, 0)$ by combining several general results from Lie groups and viscosity solution techniques. The simple structure of the abnormal set for ideal Carnot groups enable us to obtain local estimates near the identity via a comparison of the homogeneous norm.  



Our proof of the general result as stated in Theorem \ref{t2} involves more geometric aspects of Carnot groups. It utilizes the notion of so-called $C$-nearly semiconcavity,  geometric properties of the endpoint map  
and the fact that  all minimizing  geodesics in step 2 Carnot groups are normal. The notion of $C$-nearly horizontal semiconcavity was recently introduced by Badreddine and Rifford in \cite{BR18}. 
In the Euclidean setting, 
it is not equivalent to the standard notion of semiconcavity unless we 
restrict it to compact sets and allow the constants in its definition to depend on the compact set. 
For Carnot groups, we are able to prove that, while in general $C$-nearly horizontal semiconcavity is locally weaker 
than h-semiconcavity, they are indeed equivalent under the additional assumption of local Euclidean Lipschitz continuity for the function. Fortunately, such a local Lipschitz regularity does hold for the squared CC distance in step 2 Carnot groups. Therefore we can replace the local Euclidean semiconcavity with the results in  \cite{BR18} and show the h-semiconcavity of $d^2(\cdot, 0)$ for all Carnot groups of step 2.

It is worth mentioning that the assumption of step 2 is essential; in fact, as shown in  \cite{MM16}, the squared CC distance fails to be h-semiconcave in the Engel group, which is of step 3 and thus not ideal; see Proposition~\ref{cp2}. 


This new regularity property for the squared CC distance leads to various applications to the study of nonlinear PDEs in step 2 Carnot groups. We include an application to Hamilton-Jacobi equations in Section \ref{sec:app}. For a class of time-dependent convex Hamilton-Jacobi equations in step 2 Carnot groups, we show the spatial h-semiconcavity of viscosity solution that is given by the Hopf-Lax formula (associated to the CC-metric). The h-semiconcavity constant we obtained depends on $t>0$ but is independent of the space variables. Our result provides a sub-Riemannian generalization of the Euclidean counterpart; see for example \cite[Theorem 1.6.1]{CS04} for the spatial semiconcavity of the Hopf-Lax solution to Hamilton-Jacobi equations in the Euclidean space.  

Although we have proved that the square of the CC distance is h-semiconcave, its regularity turns out to be still very different from the Euclidean case. It is obvious that the squared Euclidean distance from the origin is convex in the space. In contrast, we will show that, the squared CC distance fails to be h-semiconvex in the Heisenberg group; see Proposition \ref{cp1}. More related discussions will be given in a forthcoming paper \cite{LZZ23}. 

The paper is organized as follows: In Section \ref{sec:carnot} we go over some basics about Carnot groups including the group multiplication, the dilation and the CC distance.  We recall some known (local) inclusions between the CC balls and the Euclidean balls, and between the CC distance and the homogeneous distance. 
The notions and properties of the endpoint map, normal and abnormal geodesics, and ideal Carnot groups are reviewed in Section \ref{sec:id-carnot}. 
Some related details for the special case of the Heisenberg group as well as the notions of the horizontal gradient and the horizontal Laplacian (or sub-Laplacian) are provided in these two sections as well. 
In Section \ref{sec:h-concavity} we go over the notions and basic properties of h-concave and h-semiconcave functions. 

Section \ref{sec:pf} is devoted to the proof of Theorem \ref{t1} in the case of ideal Carnot groups. We also disprove the h-convexity of the squared CC distance. 
In Section \ref{sec:step2} we prove Theorem \ref{t2}, extending Theorem \ref{t1} to the case of all Carnot groups of step 2. We also show that the assumption of step 2 is necessary for the result to hold. 
Some related consequences and generalizations are collected in Section \ref{sec:rc}. We provide our applications to Hamilton-Jacobi equations in Section \ref{sec:app}, showing the h-semiconcavity of Hopf-Lax solutions in space under suitable assumptions for the Hamiltonian. 

\subsection*{Acknowledgments}
The authors would like to thank Luca Rizzi for providing helpful comments on the first version of the manuscript and bringing to our attention the notion of $C$-nearly horizontal semiconcavity in \cite{BR18}.

The work of QL was supported by JSPS Grant-in-Aid for Scientific Research (No.~19K03574, No.~22K03396).  The work of YZ was supported by JSPS Grant-in-Aid for Early-Career Scientists (No.~24K16928).

\section{Preliminaries}\label{sec:pre}
\setcounter{equation}{0}

\subsection{Carnot groups}\label{sec:carnot}

We begin with some basic facts about Carnot groups. For more details, we refer to \cite{BLU07}.

\begin{definition}[Carnot group]
\label{Carnot}
A  {\it Carnot group} is a connected and simply connected Lie group $\G$ whose Lie algebra $\g$ 
has a stratification $\g = \bigoplus_{j=1}^s \g_j$, that is, a linear splitting $\g = \bigoplus_{j=1}^s \g_j$ where $[\g_1,\g_j]=\g_{j+1}$ for $j =1,\dots,s-1$ and $[\g_1,\g_s]=\{0\}$. If $\g_s \ne \{0\}$, the number $s$ is called the {\it step} of $\G$. 
\end{definition}
Note that the case $s=1$ coincides with the standard Euclidean space, therefore, here we will always consider the case $s\geq 2$.

By using the exponential map, we can always identify a Carnot group $\G$ with its Lie algebra $\g$ with the group law given by the so-called Baker-Campbell-Dynkin-Hausdorff formula; see \cite[\S~15]{BLU07} for more details. Furthermore, by choosing a suitable basis of $\g$ consisting of bases of $\g_j$, it can be further identified with $(\R^n, \cdot)$ with $\R^n = \R^{n_1} \times \ldots \times \R^{n_s}$, where  $\cdot$ is a non commutative operation. Here $n=n_1+\dots+n_s$ denotes the topological dimension of $\G$ as a manifold and $n_j$ represents the dimension of $\g_j$. After this identification, the group identity becomes $0$ and $p^{-1} = -p$.  For $r > 0$, writing $p = (p^{(1)}, \ldots, p^{(s)}) \in \R^n \cong \R^{n_1} \times \ldots \times \R^{n_s}$, we can define the {\it dilation} $\delta_r$ on $(\R^n, \cdot)$ by
\[
\delta_r(p^{(1)}, \ldots, p^{(s)}) := (r p^{(1)}, \ldots, r^s p^{(s)}),
\]
which is an automorphism of $(\R^n, \cdot)$. Note that the dilations defined above are anisotropic; for a more formal definition of the dilations defined on Carnot groups, we refer to \cite{BLU07}.

 Moreover, the group multiplication satisfies
\begin{align}\label{BCDH}
p \cdot q = p + q + \cR(p,q), \qquad \forall \, p,q \in \R^n, 
\end{align}
with $\cR = (\cR^{(1)}, \ldots, \cR^{(s)}) \in  \R^{n_1} \times \ldots \times \R^{n_s}$, $\cR^{(j)}$ a polynomial depending only on the first $n_1 + \ldots + n_{j - 1}$ variables of $p$ and $q$, i.e. the variables associated  via the exponential map with the first $j - 1$ layer of the Lie algebra. In particular, when the step $s = 2$, we have
\begin{align}\label{BCDH2}
\cR(p,q) = (0, \mathcal{B}(p^{(1)},q^{(1)})) \in \R^n \cong \R^{n_1} \times \R^{n_2}, \qquad \forall \, p,q \in \R^n,
\end{align}
for some skew-symmetric bilinear form $\mathcal{B}: \R^{n_1} \times \R^{n_1} \to \R^{n_2}$. 
For this identification, more details can be found in \cite[Proposition 2.2.22 and \S~3.2]{BLU07}.

\begin{remark}
Given $\G$ step two Carnot group, let us consider $\cR(p,q)$ introduced in \eqref{BCDH2}, there exists a constant $C_0>0$ such that
\begin{equation}
\label{C_0}
|\cR(p,q)| = |\mathcal{B}(p^{(1)}, q^{(1)})| \le C_0 |p^{(1)}||q^{(1)}| \le C_0 |p| |q|,  \quad \forall \, p, q \in \G,
\end{equation}
where $|\cdot|$ is the norm on $\G \cong \R^n$.
\end{remark}

For the sake of simplicity, we use $m := n_1 = \mathrm{dim} \g_1$ and write $p^{(1)} = (p_1, \ldots, p_m)$. 
\begin{definition}
\label{VectorFields}
For $1 \le i \le m$, we use $\X_i \in \g_1$ to denote the left-invariant vector field on $\G$ which coincides with $\frac{\partial}{\partial p_i}$ at the identity. Note that $\{\X_1,\ldots, \X_m\}$ forms a basis of $\g_1$.
\end{definition}

\begin{example}\label{e1}
The {\it Heisenberg group} $\H = \R^3 \cong \R^2 \times \R$ is the simplest Carnot group  whose group law \eqref{BCDH}
 is given by 
\[
(x,y,z) \cdot (\t{x},\t{y},\t{z}) = \left( x + \t{x}, y+ \t{y}, z + \t{z} + \frac{1}{2} (x \t{y} - \t{x} y)\right),
\]
which means that $\mathcal{B}((x,y),(\t{x},\t{y})) = \frac{1}{2} (x \t{y} - \t{x} y)$ in \eqref{BCDH2}. From the group multiplication defined above, it is easy to see that the center of $\H$ (the set of elements which can commute with all the other elements) is $\{0\} \times \R$.
The basis of $\g_1$ is given by
\begin{equation}
\label{vectorFieldsHeisenberg}
\X_1 = \frac{\partial}{\partial x} - \frac{y}{2} \frac{\partial}{\partial z}, \qquad \X_2 = \frac{\partial}{\partial y} + \frac{x}{2} \frac{\partial}{\partial z}.
\end{equation}
Let $\Y = \frac{\partial}{\partial z}$. The Lie algebra $\g$ of $\H$ is given by $\g = \g_1 \oplus \g_2$ with 
\[
\g_1 = \mathrm{span} \{\X_1, \X_2\}, \qquad \g_2 = \mathrm{span} \{\Y\}.
\]
and the only nontrivial bracket relation of $\g$ is $[\X_1, \X_2] = \Y$. In this  particular case of the Heisenberg group, one can show easily that \eqref{C_0} holds with $C_0=1$. In fact, we have 
\[
 |\mathcal{B}(p^{(1)}, q^{(1)})|=\frac{1}{2}|x \t{y}-\t{x}y|\leq \frac{1}{2}  \big(|x||\t{y}|+|\t{x}||y|\big) \leq \, |p^{(1)}||q^{(1)}|\,,
\]
for $p^{(1)} = (x,y)$ and $q ^{(1)}= (\t{x},\t{y})$.

\end{example}

Carnot groups are Lie groups, therefore they have also a manifold structure. Next we  briefly introduce it and we refer to  \cite[\S~4.5]{M02} 
for more details.

\begin{definition}[Sub-Riemannian structure]
\label{Sub-Riemannian structure}
On a Carnot group $\G$, the canonical left-invariant sub-Riemannian structure $(\D,g)$ is defined as follows: the horizontal distribution $\D$ (a sub-bundle of the tangent bundle $T\G$ satisfying the bracket generating condition) is generated by $\g_1$ and the metric $g$ on $\D$ is determined by $\{\X_1,\ldots, \X_m\}$. To be more precise, we have
\[
\D_p := \mathrm{span} \{\X_1(p),\ldots, \X_m(p)\} ,
\]
and $\{\X_1(p),\ldots, \X_m(p)\}$ forms an orthonormal basis at every point $p \in \G$. 
\end{definition}

An {\it horizontal path} is an absolutely continuous map $\gamma : [0,1] \to \G$ such that $\dot{\gamma}(\tau) \in \D_{\gamma(\tau)}$ for a.e. $\tau$, whose {\it length} can be calculated by 
\begin{equation}
\label{lenght}
\ell(\gamma) := \int_0^1 \sqrt{g(\dot{\gamma}(\tau),\dot{\gamma}(\tau))} \, d\tau.
\end{equation}
\begin{definition}[Carnot-Carath\'eodory distance]
\label{C-C distance}
The {\em Carnot-Carath\'eodory distance} (or in short {\it CC distance}) between two points $p, q \in \G$ is defined as
\[
d(p,q) := \inf\{\ell(\gamma)\,| \, \gamma \mbox{ horizontal}, \gamma(0) = p, \gamma(1) = q\}.
\]
\end{definition}
By the celebrated Chow-Rashevsky Theorem (see for example \cite{ABB20, M02}), $d$ is a well-defined finite distance and its induced topology is the same as the manifold topology; in particular, in Carnot groups, this means that $d$ is continuous with respect to the standard Euclidean topology. Therefore $(\G,d)$ is a metric space. We call the sub-Riemannian structure $(\D,g)$  {\it complete} if it is complete as a metric space. A {\it (length) minimizing geodesic} is a horizontal path $\gamma$ such that $\ell(\gamma) = d(\gamma(0), \gamma(1))$.

In addition, the following properties of the CC distance hold (cf. \cite[Proposition 5.2.6]{BLU07}):
\begin{equation}
\label{homo}\begin{aligned}
&d(\delta_r(p), \delta_r(q)) = r d(p,q), \quad\qquad\qquad \forall \, r > 0, p , q \in \G,\\
&d(p, q) = d(q^{-1}\cdot p,0) =d(p^{-1}\cdot q, 0), \quad\, \forall \, r > 0, p , q \in \G.
\end{aligned}
\end{equation}

In the following we use $B_E(p,r)$ and $B_{CC}(p,r)$ to denote, respectively, the Euclidean ball and the CC ball centred at $p \in \G$ with radius $r > 0$, i.e.
\begin{align}\label{ball}
B_E(p,r) := \{q \in \G\,| \, |p - q| < r\}, \quad 
B_{CC}(p,r) := \{q \in \G\,| \, d(p,q) < r\}.
\end{align}
A well-known relation between the Euclidean distance and the CC distance is as follows.

\begin{proposition}[\cite{AES85}, Proposition 1.1] \label{pcom}
On a Carnot group $\G$ with step $s$, for every compact set $K \subset \G$, there exists a constant $C(K) > 0$ such that
\[
\frac{1}{C(K)} |p - q| \le d(p,q) \le C(K) |p - q|^{\frac{1}{s}}, \quad \forall \, p,q \in K.
\] 
\end{proposition}
It immediately implies the following local inclusions between the Euclidean balls and the CC balls, i.e. given any compact set $K \subset \G$,  if  $B_{CC}(p,r) \subset K$, then
\begin{align}\label{comb}
B_E(p,C(K)^{-s}r^s) \subset B_{CC}(p,r) \subset B_E(p,C(K)r).
\end{align}

In order to simplify our notation, we denote by $d_0(p)$ the CC distance between the point $p$ and the group identity $0$, i.e.  $d_0(p)=d(p,0)$. It is worth mentioning the following equivalence between the CC distance from the identity and the homogeneous norm.

\begin{proposition}[\cite{BLU07}, Proposition 5.1.4] \label{phomo}
Let $d$ be the CC distance of a Carnot group and $d_0=d(\cdot, 0)$. Then, there exists a constant $C \ge 1$ such that
\[
C^{-1} |p|_\G \le d_0(p) \le C |p|_\G, \quad \forall \, p \in \G,
\]
where the {\it homogeneous norm} $|\cdot|_\G$ is defined by
\begin{equation}
\label{homogeneousNorm}
|p|_\G = \left( \sum_{i = 1}^s |p^{(i)}|^{\frac{2s!}{i}}\right)^{\frac{1}{2s!}}, \qquad \forall \, p = (p^{(1)}, \ldots ,p^{(s)}) \in \G,
\end{equation}
with $s$ the step of the group $\G$ and $p^{(i)}$ associated by the exponential map to the $i$-layer $\g_i$.
\end{proposition}

We remark that in the special case of Heisenberg group $\H$ (see Example \ref{e1}), the homogeneous norm is 
\[
|(x,y, z)|_{\H} = \left( (x^2 + y^2)^2 + z^2 \right)^{\frac{1}{4}}, \qquad \forall \, (x,y,z) \in \H,
\]
which corresponds to \eqref{homogeneousNorm} with $s=2$, $p = (x,y,z)$, $p^{(1)}=(x,y)$ and $p^{(2)}=z$.

\subsection{Endpoint map and ideal Carnot groups}\label{sec:id-carnot}

Next we introduce the endpoint and ideal Carnot groups.
More details on this part can be found in \cite[\S~8]{ABB20} and \cite{M02, R14}.
We recall that it is usually more convenient to minimize the energy $J$ as below rather than the original length $\ell$ defined in \eqref{lenght}:
\begin{align} \label{defJ}
J(\gamma):= \frac{1}{2} \int_0^1 g(\dot{\gamma}(\tau),\dot{\gamma}(\tau)) \, d\tau.
\end{align}
In fact, we have the following relation (see e.g. \cite[Theorem 1.1.7]{M01} or \cite[Proposition 2.1]{R14}):
\begin{equation}
\label{CauchProblem}
\frac{1}{2}d^2(p,q) = \inf\{J(\gamma)\,| \, \gamma : [0,1] \to \G, \mbox{ horizontal}, \gamma(0) = p, \gamma(1) = q\}.
\end{equation}
For a fixed $p \in \G$ and any control $u \in L^2([0,1],\R^m)$, let $\gamma_u$ be the unique maximal solution of the following Cauchy problem:
\begin{equation}\label{CauchyProblem}
\dot{\gamma}_u(\tau) = \sum_{j = 1}^m u_j(\tau) \X_j(\gamma_u(\tau)), \qquad \gamma(0) = p.
\end{equation}
\begin{definition}[Endpoint map]
\label{endpointMap} We use $\U_p$ to denote the  set of $u \in L^2([0,1],\R^m)$ such that the corresponding
trajectories $\gamma_u$ solving \eqref{CauchyProblem}
starting at $p$ are defined on the interval $[0, 1]$. $\U_p$ is an open set in $L^2([0,1],\R^m)$. 
The {\it endpoint map based at $p$} is the map $\mathcal{E}_p: \U_p \to \G$ defined as
\[
\mathcal{E}_p(u) := \gamma_u(1).
\]
\end{definition}
We then obtain an {\it energy functional} on $\U_p$ given by 
\begin{equation}\label{enf}
\mathcal{J}(u) := J(\gamma_u) = \frac{1}{2} \int_0^1 |u(\tau)|^2 d\tau. 
\end{equation}
Note that $\mathcal{E}_p$ is a smooth function on $\U_p$ (cf. \cite[Appendix D]{M02}) and a length minimizing geodesic joining $p$ and $q$ is just $\gamma_u$ with $u$ minimizing $\mathcal{J}$ under the constraint $\mathcal{E}_p(u) = q$. Thus, by the method of Lagrange multipliers (see e.g. \cite[Theorem B.2]{R14} or \cite[\S~8.2]{ABB20}), for such $u$, there exists a non-trivial pair $(\lambda, \nu)$, such that
\begin{equation}
\label{nu}
\lambda \circ D \mathcal{E}_p (u) = \nu D \mathcal{J} (u), \qquad \lambda \in T_q^* \G, \ \nu \in \{0,1\}, 
\end{equation}
here $\circ$ denotes the composition and $D$ the (Fr\'echet) differential. 
\begin{definition}
[Normal and abnormal minimizing geodesic]
\label{Normal and abnormal minimizing geodesic}
Given $\nu$ introduced in \eqref{nu}, the length minimizing geodesic $\gamma_u$ is called {\em normal} if $\nu = 1$ and {\em abnormal} if
$\nu = 0$. 
\end{definition}
We remark that a minimizing geodesic could be both normal and abnormal at the same time since the pair $(\lambda, \nu)$ is not necessarily unique (see  \cite[\S~5.3.3]{M02} and also Example \ref{e2} below).  
We call an abnormal minimizing geodesic {\em trivial } if it is a constant curve.

Now we define the {\it sub-Riemannian Hamiltonian function} $H : T^* \G \to \R$ as 
\[
H(\mu) := \frac{1}{2} \sum_{j = 1}^m (\mu \circ \X_j)^2, \qquad \forall \, \mu \in T^* \G
\]
and use $\vec{H}$ to denote the {\it Hamiltonian vector field} given by $H$ with respect to the canonical symplectic structure on $T^*\G$. If $\gamma_u$ is normal, then from \cite[Proposition 8.9]{ABB20} we know that there exists a curve $\mu: [0,1] \to T^* \G$ such that $\mu$ satisfies the Hamilton equation 
\begin{align}\label{sRH}
\dot{\mu}(t) = \vec{H}(\mu(t)), \quad \forall \, t \in [0,1], 
\end{align}
$\mu(1) = \lambda$, and $\pi(\mu) = \gamma_u$, where $\lambda$ is the one in \eqref{nu} (with $\nu = 1$) and $\pi : T^*\G \to \G$ is the projection of the bundle. Such curve $\mu$ is called a {\it normal extremal}.

The next definition introduces the groups under consideration in this paper; we refer to \cite{R13,R14} for more details.
\begin{definition}[Ideal Carnot group]
Given a Carnot group $\G$, we say that $\G$ is {\it ideal} if the sub-Riemannian structure $(\D,g)$, introduced in Definition \ref{Sub-Riemannian structure}, is ideal, which by definition means that it is complete and it does not admit non-trivial abnormal minimizing geodesics.
\end{definition}
The following notion of fatness can help us check whether a Carnot group $\G$ is ideal or not. Recall that $\G$ is called {\it fat} if for every $p \in \G$ and $\X \in \D$ with $\X(p) \ne 0$, we have 
\[
\D_p + [\D,\X]_p = T_p \G.
\]
Thanks to the left invariance of the sub-Riemannian structure $(\D,g)$ on $\G$, the property above is equivalent to saying that for every $\X \in \g_1 \setminus \{0\}$, the following holds true:
\begin{align}\label{idrel}
\g_1 + [\g_1, \X] = \g.
\end{align}
By \cite[Theorem 10]{R16}, a Carnot group is ideal if and only if it is fat, which trivially implies that $\G$ is step two, i.e., $s = 2$. 

\begin{example}\label{e2}
The Heisenberg group $\H$ appearing in Example \ref{e1} is ideal. The simplest non-ideal Carnot group is $\R \times \H$. To be more precise, suppose that $\T$ is a nonzero constant vector field on $\R$ and the Lie algebra of $\R \times \H$ is given by $\g = \g_1 \oplus \g_2$ with 
\[
\g_1 = \mathrm{span} \{\T, \X_1, \X_2\}, \qquad \g_2 = \mathrm{span} \{\Y\}.
\]
Since $\T$ commutes with $\X_1,\X_2,\T$, \eqref{idrel} fails for $\X = \T$, which implies that this Carnot group is not ideal. To be more precise, we can write down explicitly a non-trivial abnormal geodesic. To show a length minimizing geodesic $\gamma_u$ is normal or abnormal, by definition we need to show \eqref{nu} with the endpoint map defined in Definition \ref{endpointMap} and the energy functional defined in \eqref{enf}. Now in our case, by considering $\X_1$, $\X_2$ given in \eqref{vectorFieldsHeisenberg} and 
$\T = \frac{\partial}{\partial w}$,
 and  writing 
 \eqref{CauchyProblem} in coordinates $(w,x,y,z)$,
 we have
\[
\dot{\gamma}_u(\tau) := \begin{pmatrix}
\dot{\gamma}_u^1(\tau) \\
\dot{\gamma}_u^2(\tau) \\
\dot{\gamma}_u^3(\tau) \\
\dot{\gamma}_u^4(\tau)
\end{pmatrix}
= \begin{pmatrix}
u_1(\tau) \\
u_2(\tau) \\
u_3(\tau) \\
\frac{1}{2}[-u_2(\tau) \gamma_u^3(\tau)+ u_3(\tau) \gamma_u^2(\tau)]
\end{pmatrix}.
\]
Assuming that we start from the group identity 0, we obtain
\[
\gamma_u^j(\tau) = \int_0^\tau u_j(s) ds, \qquad \forall\, j = 1,2,3,
\]
and 
\[
\gamma_u^4(1) = \frac{1}{2} \int_0^1 [-u_2(\tau) \gamma_u^3(\tau)+ u_3(\tau) \gamma_u^2(\tau)] d\tau 
= \frac{1}{2} \int_0^1 \int_0^\tau [-u_2(\tau) u_3(s)+ u_3(\tau) u_2(s)] ds d\tau.
\]
This gives the formula of endpoint map based at the group identity in $\R \times \H$
\[
\mathcal{E}_0(u) = \begin{pmatrix}
\int_0^1 u_1(\tau) d\tau \\
\int_0^1 u_2(\tau) d\tau \\
\int_0^1 u_3(\tau) d\tau \\
\frac{1}{2} \int_0^1 \int_0^\tau [-u_2(\tau) u_3(s)+ u_3(\tau) u_2(s)] ds d\tau
\end{pmatrix}.
\]
Taking the differential, 
we obtain
\[
D \mathcal{E}_0(u)v = \begin{pmatrix}
\int_0^1 v_1(\tau) d\tau \\
\int_0^1 v_2(\tau) d\tau \\
\int_0^1 v_3(\tau) d\tau \\
\frac{1}{2} \int_0^1 \int_0^\tau [-v_2(\tau) u_3(s)+ v_3(\tau) u_2(s) -u_2(\tau) v_3(s)+ u_3(\tau) v_2(s)] ds d\tau
\end{pmatrix}, 
\]
since we can check directly that
\[
\mathcal{E}_0(u + v) - \mathcal{E}_0(u) - D \mathcal{E}_0(u)v = o(\|v\|_{L^2})
\]
as $\|v\|_{L^2} \to 0+$.
Similarly, we can take the differential in the direction $v \in L^2([0,1],\R^3)$ for the energy $\mathcal{J}$ as \eqref{enf} and the result is
\[
D\mathcal{J}(u)v = \int_0^1 \langle u(\tau), v(\tau) \rangle d\tau.
\]
Now we consider the length minimizing geodesic $\gamma(\tau) = \gamma_u(\tau) = (\tau,0,0,0), \tau \in [0,1]$ with $u \equiv (1,0,0)$. Using the formulas above, it is not hard to check that
\[
\langle \lambda,  D\mathcal{E}_0 (u) \rangle = 0, \qquad \langle \widetilde{\lambda},  D\mathcal{E}_0(u) \rangle = D\mathcal{J}(u)
\]
with $\lambda = (0,0,0,1)$ and $\widetilde{\lambda} = (1,0,0,0)$. Here $\langle \cdot, \cdot \rangle$ is the inner product on the Euclidean space $\R^4$ inducing the norm  $|\cdot|$. As a result, $\gamma$ is a non-trivial minimizing geodesic that is both normal and abnormal.
\end{example}

In step 2 Carnot groups every abnormal minimizing geodesic needs to be also normal (see for example \cite[Theorem 2.22]{R14}) while this is not anymore true for step 3 (or higher) Carnot groups. Moreover,
Carnot groups of step  $\geq 3$ are never ideal Carnot groups because they are not fat. 
\begin{example}\label{e3}
 The {\em Engel group} $\E = \R^4 \cong \R^2 \times \R \times \R$, which is the simplest step 3 Carnot group, is not ideal. To be more precise, the Engel group can be identified with $\R^4$ with the following multiplication: given $p=(x,y,z,s),\t{p}=(\t{x},\t{y},\t{z},\t{s})\in\E$ 
\[
p\cdot \t{p}=\left( x + \t{x}, y+ \t{y}, z + \t{z} + \frac{1}{2} (x \t{y} - \t{x} y), s + \t{s} + \frac{1}{2} (x \t{z} - \t{x} z) + \frac{1}{12} (x - \t{x})(x \t{y} - \t{x} y)\right).
\]
Defining
\begin{align*}
\X_1 = \frac{\partial}{\partial x} - \frac{y}{2} \frac{\partial}{\partial z} - \left(  \frac{xy}{12}+ \frac{z}{2}\right) \frac{\partial}{\partial s}, &\qquad \X_2 = \frac{\partial}{\partial y} + \frac{x}{2} \frac{\partial}{\partial z} + \frac{x^2}{12} \frac{\partial}{\partial s}, \\
\Y = \frac{\partial}{\partial z} + \frac{x}{2} \frac{\partial}{\partial s}, & \qquad \Z = \frac{\partial}{\partial s}.
\end{align*}
The Lie algebra $\g$ of $\E$ is given by $\g = \g_1 \oplus \g_2 \oplus \g_3$ with
\[
\g_1 = \mathrm{span} \{\X_1, \X_2\}, \qquad \g_2 = \mathrm{span} \{\Y\}, \qquad \g_3 = \mathrm{span} \{\Z\}.
\]
Here the nontrivial bracket relations of $\g$ are $[\X_1, \X_2] = \Y$ and $[\X_1,\Y] = \Z$. For more details on the Engel group, we refer to \cite{AT13, AS11, AS13, AS15} and the references therein.
\end{example}

We conclude this subsection with 
the definition of several differential operators in Carnot groups, which utilizes the derivatives  along the vector fields introduced in Definition \ref{VectorFields}. The {\it horizontal gradient} of $\varphi \in C^1(\G)$ at the point $p \in \G$, denoted by $\nabla_H \varphi (p)$, is defined by
\[
\nabla_H \varphi (p) = (\X_1 \varphi (p), \ldots, \X_m \varphi (p)) \in \R^m.
\]
The {\it horizontal plane $\HH_0$ passing through the identity $0$}  is a subspace of $\G \cong \R^n$ defined by $\R^m \times \{0\} \times \ldots \times \{0\}$. It is clear that $\HH_0$ is isomorphic to $\R^m$ and thus from now on we will not distinguish $\HH_0$ from $\R^m$. 
If we use $\langle \cdot, \cdot \rangle$ to denote the inner product on the Euclidean space $\G \cong \R^n$ inducing the norm  $|\cdot|$, with these notations, it is not difficult to see that
\begin{align} \label{firt}
\langle \nabla_H \varphi (p),h \rangle = \frac{d}{d\tau}\bigg|_{\tau = 0} \varphi(p \cdot \tau h), \quad \forall \, h \in \HH_0.
\end{align}
Moreover, similar to the definition of the horizontal gradient, the (symmetrized) {\it horizontal Hessian} of $\varphi \in C^2(\G)$ at the point $p \in \G$, denoted by $(\nabla_H^2 \varphi (p))^*$, is the unique $m \times m$ symmetric matrix satisfying the following formula:
\begin{align} \label{sect}
\langle (\nabla_H^2 \varphi (p))^* h,h \rangle = \frac{d^2}{d\tau^2}\bigg|_{\tau = 0} \varphi(p \cdot \tau h), \quad \forall \,  h \in \HH_0.
\end{align}
The {\it (canonical) sub-Laplacian} at $p\in \G$ is thus defined by $\Delta_H u (p)= \sum_{i = 1}^m \X_i^2 u(p)$.

\subsection{H-concavity and h-semiconcavity}\label{sec:h-concavity}

On Carnot groups it is possible to introduce several notions of convexity/concavity. However, some of them exhibit rather unusual behavior in this sub-Riemannian setting. For example, in \cite{MR05}, the authors showed that geodetical convexity is not a suitable notion on Heisenberg group in the sense that the family of geodetically convex sets only consist of the whole group, the empty set and the geodesics, and as a consequence the only geodetically convex functions in the Heisenberg group are the constant functions. Furthermore, strong H-convexity is considered in \cite{DGN03, CCP07}, and it turns out that it is still a too restrictive notion. 
A notion of horizontal convexity, shortly h-convexity, was introduced,  independently, by Lu, Manfredi and Stroffolini \cite{LMS03}, for the Heisenberg group, and by Danielli, Garofalo and Nhieu  \cite{DGN03} in more general Carnot groups. Later the notion has been discussed in various papers e.g. \cite{BR03,  GT06, M06, M062, SY06, JLMS07}. 
In more recent years h-convexity has been applied to study properties of solutions for certain classes of PDEs \cite{LMZ16, LZ21, LZZ23}. Concerning how h-convexity can be applied to sets, see \cite{DF19} for an overview and \cite{KLZ24, KLZZ24} for further applications related to nonlinear PDEs. The authors of \cite{BD11, BD14} extended the notion of  the h-concavity the setting of vector fields and further studied the notion of h-semiconcavity, which is the key topic of the present work.\\

For the purpose of this paper, we introduce directly the notion of h-concavity, holding the usual relation that $u$ is h-convex if and only if $-u$ is h-concave.
\begin{definition}[H-concavity]
\label{h-concavity}
Given an open set $\Omega \subset \G$, a function $u \in LSC(\Omega)$ (i.e. the function is lower semicontinuous in $\Omega$)
 is called  {\em h-concave} if and only if
\begin{equation}\label{defhc}
\begin{aligned}
u(p \cdot h)  + & u(p \cdot h^{-1}) - 2 u(p) \le 0, \\
& \forall \, p \in \Omega, h \in \HH_0 \text{ such that $[p \cdot h^{-1}, p \cdot h] := \{p \cdot \tau h \, | \, \tau \in [-1,1]\} \subset \Omega$.
}
\end{aligned}
\end{equation}
\end{definition}

\begin{remark}
Note that for Definition \ref{h-concavity} to hold true, one could relax the assumption of lower semicontinuity.
 Nevertheless, since later we use the viscosity characterization of h-concave functions,  we ask such regularity directly in the definition. 
This definition follows from \cite{LMS03} and \cite{DGN03}, where the notion is stated in a slightly different form, but these formulations coincide with ours for LSC functions; see \cite{JLMS07} for details.
\end{remark}
Similarly to the standard Euclidean characterization established first by Alvarez, Lasry and Lions in 1997 \cite{ALL97}, the  h-concavity of a function can be characterized by the sign of its horizontal Hessian in the viscosity sense: this characterization was first introduced for the Heisenberg group
in \cite{LMS03} and later proved in Carnot groups in \cite{{W05},{JLMS07}}. See also  \cite{BD11} for the more general case of Carnot-type H\"ormander vectors fields. 
 
Combining the results from the above mentioned papers, we have 
\begin{equation}\label{vis-char}
u\; \textrm{is}\; \textrm{h-concave in $\Omega$, if and only if,}\;
-(\nabla_H^2 u)^* \ge 0 \;\textrm{in $\Omega$ holds in the viscosity sense}.
\end{equation}
 To be more precise, the viscosity inequality means that $-(\nabla_H^2 \varphi(p))^* \ge 0$ whenever there exist $\varphi \in C^2(\Omega)$ and $p \in \Omega$ such that $u - \varphi$ has a local minimum at $p$. 
 
\begin{example}
On the Heisenberg group $\H$ in Example \ref{e1}, we have $\X_1 = \partial_x - \frac{y}{2} \partial_z$,  $\X_2 = \partial_y + \frac{x}{2} \partial_z$, and the horizontal Hessian can be represented by
\[
 (\nabla_H^2 \varphi )^* = \begin{pmatrix}
\X_1^2 \varphi &  \frac{1}{2} (\X_1 \X_2 \varphi+ \X_2 \X_1 \varphi)  \\
\frac{1}{2} (\X_1 \X_2 \varphi+ \X_2 \X_1 \varphi) & \X_2^2 \varphi
\end{pmatrix}
\]
with
\begin{align}\label{HeHe}
\X_1^2 = \partial_{xx} - y \partial_{xz} + \frac{y^2}{4} \partial_{zz}, \quad 
\X_2^2 = \partial_{yy}  + x \partial_{xz} + \frac{x^2}{4} \partial_{zz}, \\
\label{HeHe2}
\frac{1}{2} (\X_1 \X_2 + \X_2 \X_1 ) = \partial_{xy} - \frac{y}{2} \partial_{yz} + \frac{x}{2} \partial_{xz} - \frac{xy}{4} \partial_{zz}.
\end{align}
In this case, one can easily verify that every Euclidean concave function in $\R^3$ is also h-concave in $\H$. The reverse however is not true, as shown in \cite{DGN03, LMZ16, DF19}.
\end{example}

Next we recall the property under investigation in this paper.
\begin{definition}[H-semiconcavity]
Given an open set $\Omega \subset \G$, we call a function $u \in LSC(\Omega)$  {\em h-semiconcave} if there exists a constant $C \ge 0$ such that
\begin{align}\label{defsc}
u(p \cdot h) + u(p \cdot h^{-1}) - 2 u(p) \le  C |h|^2, \quad \forall \, p \in \G, h \in \HH_0 \ \mbox{such that} \ [p \cdot h^{-1}, p \cdot h] \subset \Omega,
\end{align}
where we recall that $|\cdot|$ is the Euclidean norm on $\G \cong \R^n$. The constant $C$ is called {\it h-semiconcavity constant}. 
\end{definition}
This is a generalization of the notion of semiconcave functions in the Euclidean space, for which we refer to \cite{CS04}. We also refer the reader to \cite{BR18} for another notion of horizontal semiconcavity called $C$-nearly horizontal semiconcavity, which will also be used later; see Definition \eqref{Cnhs} for a precise definition. In Proposition \eqref{propCnhs}, we will also discuss the relation between these two notions of horizontal semiconcavity. 

A function $u$ is called {\it h-semiconvex} it $-u$ is h-semiconcave. 
Similarly to the characterization in \eqref{vis-char} for h-concave functions, the notion of h-semiconcavity can be characterized by a bound for the horizontal Hessian in the viscosity sense.
\begin{theorem}[Proposition 5.1 of \cite{BD11}]
\label{tp2}
Given an open set $\Omega \subset \G$ and $u\in LSC(\Omega)$, the follow statements are equivalent:
\begin{enumerate}

\item $u$ is h-semiconcave in $\Omega$ with h-semiconcavity constant $C\geq 0$.
\item  We have 
\begin{equation}
\label{viscosityCharacterization}
-(\nabla_H^2 u)^* \ge -C\; \I_m \quad \textrm{in $\Omega$ in the viscosity sense,}
\end{equation}
which means that $-(\nabla_H^2 \varphi(p))^* \ge -C \;\I_m$ whenever there exist  $\varphi \in C^2(\Omega)$ and $p \in \Omega$ such that $u - \varphi$ has a local minimum at $p$. Here $\I_m$ denotes the $m \times m$ identity matrix.

\end{enumerate}
\end{theorem} 

The result below is a direct consequence of Theorem \ref{tp2} and the stability of viscosity supersolutions with respect to the infimum. It will be useful in our later applications. Below we denote by $LSC(A)$ the set of lower semicontinuous functions in a set $A$ of a metric space. 
\begin{proposition}\label{pinf}
Let $\Omega$ be an open set of a Carnot group $\G$. Let $\{u_\alpha\}_{\alpha \in \mathcal{A}}$ be a family of h-semiconcave functions on $\Omega$. Assume that for every $\alpha \in \mathcal{A}$, the function $u_\alpha$ is h-semiconcave functions in $\Omega$ with h-semiconcavity constant $C\ge 0$ independent of $\alpha$. Suppose that
\[
u(p) := \inf_{\alpha \in \mathcal{A}} u_\alpha(p) > -\infty\quad \text{for all $p\in \Omega$}.
\]
If $u\in LSC(\Omega)$, then $u$ is also h-semiconcave in $\Omega$ with the same h-semiconcavity constant $C\geq 0$.
\end{proposition}

\section{H-semiconcavity of squared CC distance in ideal Carnot groups}\label{sec:pf}
\setcounter{equation}{0}

In this section, we present our main result in the case of ideal Carnot groups. 
\begin{theorem} \label{t1}
Let $\G$ be an ideal Carnot group with CC distance $d$. Then $d^2(\cdot, 0)$ is h-semiconcave in $\G$.
\end{theorem}

 One of the tools we will use is  the local (Euclidean) semiconcavity studied in \cite{CR08}. Recall that a function $u$ is {\it locally semiconcave} in an open set $\Omega$ if for every compact convex set $K \subset \Omega$, there exists a constant $C(K) \ge 0$ such that the following holds:
\begin{align}\label{lsemiconc}
\lambda u(p) + (1 - \lambda) u(q) - u(\lambda p + (1 - \lambda) q) \le  \lambda (1 - \lambda) 
 C(K) |p - q|^2, \quad \forall \, p , q \in K, \lambda \in [0,1]. 
\end{align}
Here the constant $2C(K)$ is called the {\it semiconcavity constant on compact set $K$}. Note that the definition is independent of the choice of the norm $|\cdot|$,  since, in the Euclidean setting, different norms are equivalent up to a multiplicative constant. Hereafter let us just use the standard norm on  $ \R^n$. 
By definition an ideal Carnot group only possesses trivial abnormal minimizing geodesics. Consequently, the abnormal set of the identity $0$, which is just the set of endpoints of abnormal minimizing geodesics starting from the identity $0$, must be $\{0\}$. It follows from \cite[Theorem 1]{CR08} (and also \cite[Theorem 5.9]{FR10}) that $d^2(\cdot, 0)$ is locally (Euclidean) semiconcave on $\G \setminus \{0\}$.  The following lemma, which will be useful in the proof of our main result, is a direct consequence.   

\begin{lemma}\label{l1}
Let $\G$ be an ideal Carnot group with CC distance $d$. Then, for $d_0=d(\cdot, 0)$,  there exist two constants $C \ge 0$ and $c > 0$ such that 
\[
 d_0^2(p + v) +   d_0^2(p - v) - 2 d_0^2(p) \le C |v|^2, \quad \forall \, p \in \partial B_{CC}(0,1), \  |v| \le c.
\]
\end{lemma}

\begin{proof}
Set $S = \partial B_{CC}(0,1)$ and $\eta = \inf_{q \in S} |q|$, that is,  $\eta$ is the Euclidean distance between the origin and the boundary of the unit CC ball. 
By \eqref{comb} with  $K = \overline{B_{CC}(0,1)}$, 
it is easy to see that $\eta>0$. Thus, for every 
$q \in S$, the compact set $\overline{B_E(q,\eta/2)} \subset \G \setminus \{0\}$;  see Figure \ref{fig2}.

\begin{figure}[htp]
	\centering
  \begin{overpic}[scale = 2.8]{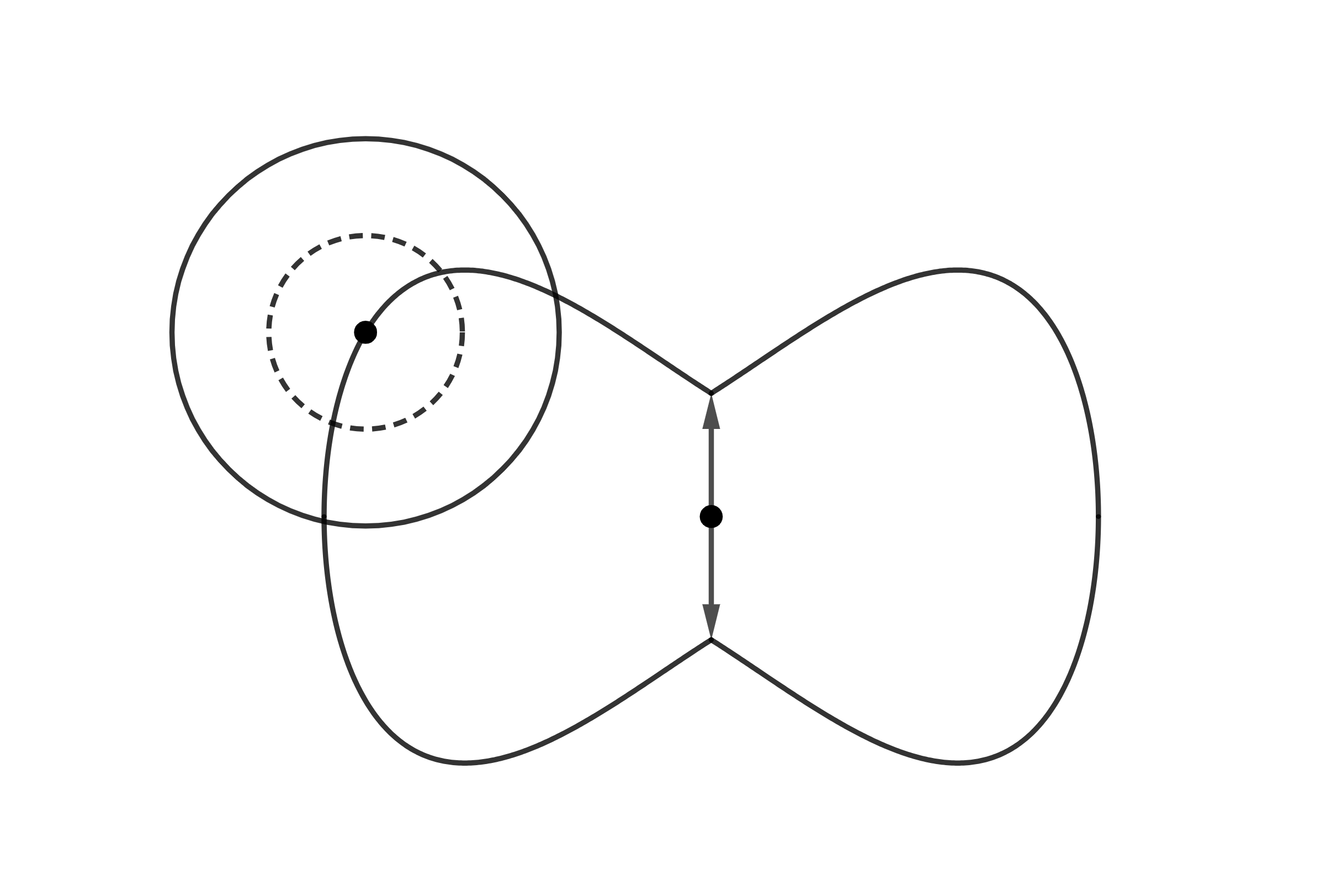}
       \put(19.5,50.5){$B_E(q,\eta/4)$}
       \put(19.5,58){$\overline{B_E(q,\eta/2)}$}
        \put(24.5,44){$q$}
        \put (55, 33){$\eta$}
        \put (55, 27){$0$}
        \put (55, 23){$\eta$}
        \put (63, 50) {$S = \partial B_{CC}(0,1)$}
   \end{overpic}
	\caption[image]{The CC ball $B_{CC}(0, 1)$ and the Euclidean ball $B_E(q, \eta/2)$}\label{fig2}
\end{figure}
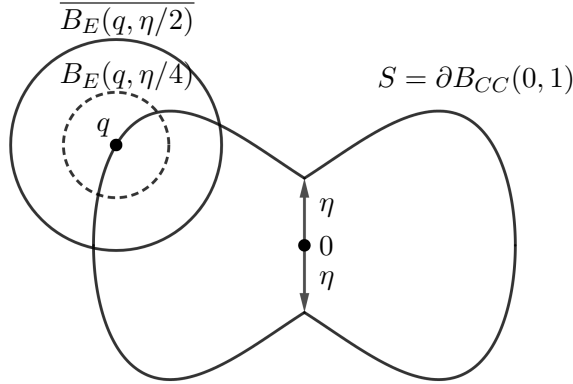
It follows from the local (Euclidean) semiconcavity of $d_0^2$ on $\G \setminus \{0\}$ (see \cite[Theorem 1]{CR08}) that there exists a $C(q,\eta) \ge 0$ such that
\begin{equation}
\label{lambdaEq}
\begin{aligned}
\lambda d_0^2(p_+) + (1 - \lambda) d_0^2(p_-) - d_0^2(\lambda p_+ + (1 - \lambda) p_-) \le & \  \lambda (1-\lambda)C(q,\eta) |p_+-p_-|^2, \\ & \quad \forall \, p_+ , p_- \in \overline{B_E(q,\eta/2)}, \lambda \in [0,1]. 
\end{aligned}
\end{equation}
As a result, choosing $p_+= p + v$, $p_-= p - v$,  and $\lambda = \frac{1}{2}$, we apply \eqref{lambdaEq} to deduce 
\[
d_0^2(p + v) + d_0^2(p - v) - 2 d_0^2(p) \le 2 C(q,\eta)  |v|^2, \quad \forall \, p \in B_E(q,\eta/4), |v| \le \eta/4.
\]
See Figure \ref{fig0}.

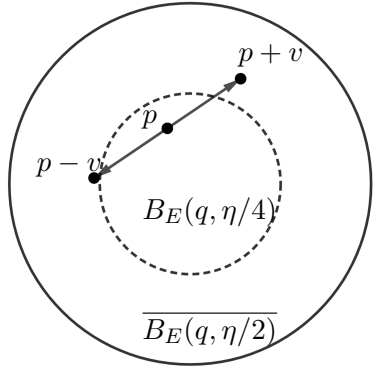
\begin{figure}[htp]
	\centering
    \begin{overpic}[scale = 0.34]{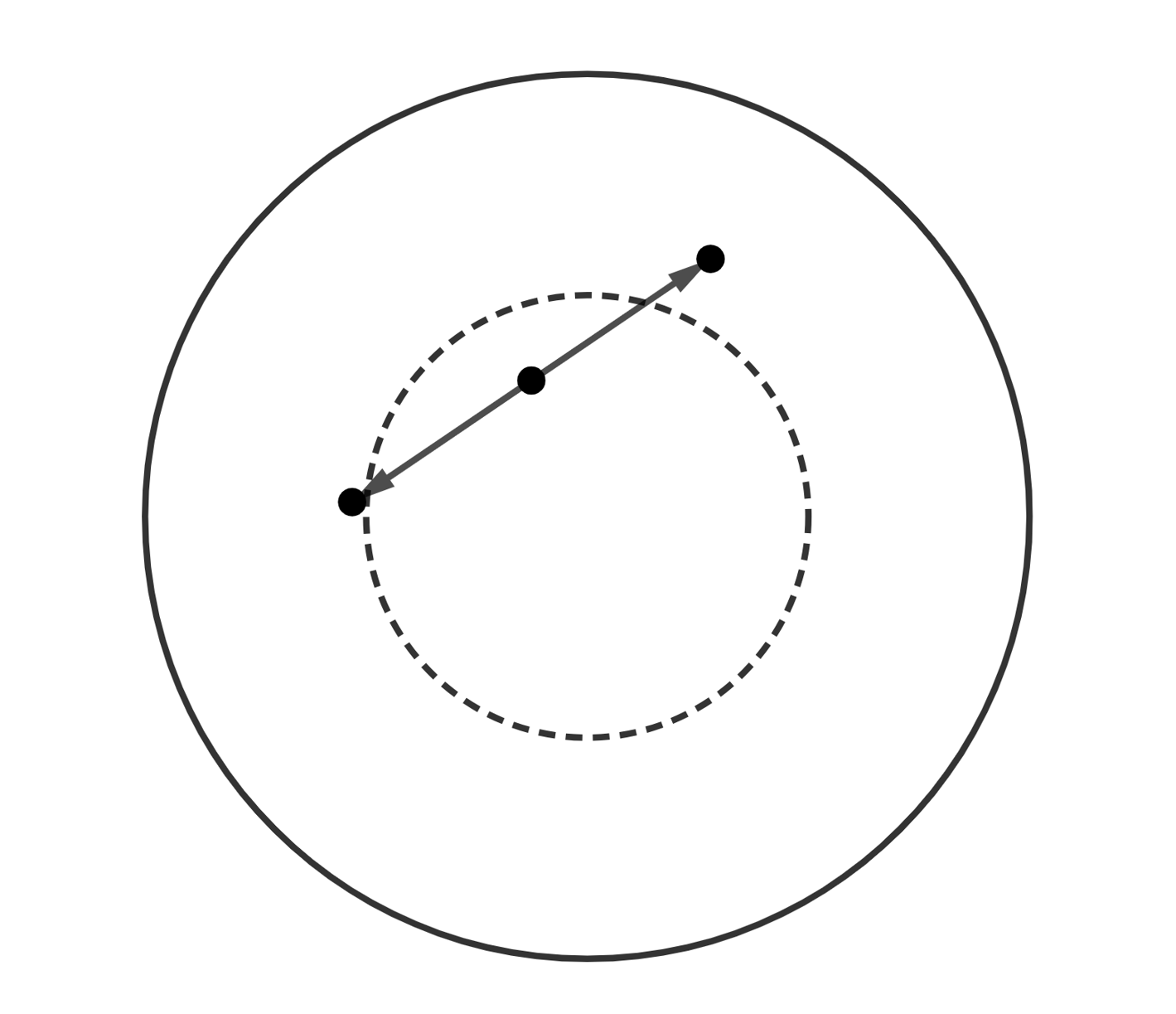}
       \put(40,35){$B_E(q,\eta/4)$}
       \put(40,10){$\overline{B_E(q,\eta/2)}$}
\put(40,55){$p$}
\put(18,45){$p - v$}
\put(60,68){$p + v$}
    \end{overpic}
	\caption[image]{The points $p_\pm$ and the Euclidean balls centered at $q$}\label{fig0}
\end{figure}

Since $\{B_E(q,\eta/4)\}_{q \in S}$ is an open cover of $S$, by compactness, there exists a finite cover $\{B_E(q_i,\eta/4)\}_{1 \le i \le N}$ with $N < +\infty$. Now we introduce
\[
C := 2 \max_{1 \le i \le N} C(q_i,\eta) \ge 0,
\]
then, for every $p \in S$, there exists an $i \in \{1, \ldots, N\}$ such that $p \in B_E(q_i,\eta/4)$, which yields 
\[
d_0^2(p + v) + d_0^2(p - v) - 2 d_0^2(p) \le 2 C(q_i,\eta)  |v|^2 \le C |v|^2, \quad \forall \,  |v| \le \eta/4.
\]
 By taking $c := \eta/4 > 0$, we complete the proof of the lemma.
\end{proof}

We stress that both constants $C, c>0$ in Lemma \ref{l1} depend only on $S = \partial B_{CC}(0,1)$, therefore they are both universal constants.

Let us now prove Theorem \ref{t1}. 
\begin{proof}[Proof of Theorem \ref{t1}]
We continue using $d_0$ to denote the CC distance from the identity. Theorefore we have $d_0^2(p) = d^2(p,0)$ for $p\in \G$. Moreover, since Carnot groups satisfy the H\"ormander condition, meaning the distribution associated to Carnot groups is bracket generating,  the CC distance $d$ is continuous   (see \cite[Theorem 2.2 and Theorem 2.3]{M02}). Therefore it is clear that the condition $d_0^2\in LSC(\G)$. 
 
We first pose the following {\bf claim:}
\begin{equation}\label{locp}
\begin{aligned}
&\text{For every $p \in \G$, there exist $C > 0$  (independent of $p$) and  $c(p) > 0$ such that}\\
& d_0^2(p \cdot h) + d_0^2(p \cdot h^{-1}) - 2 d_0^2(p) \le C |h|^2, \quad \forall \, h \in \HH_0, |h| \le c(p).
\end{aligned}
\end{equation}

Assuming that \eqref{locp} holds, we can use the viscosity characterization for h-semiconcave  functions given in Theorem \ref{tp2} to easily conclude. In fact, take $\varphi \in C^2(\G)$ such that  $d_0^2 - \varphi$ has a local minimum at some $p\in \G$. Without loss of generality, by standard viscosity theory techniques, we can assume that the local minimum is equal to 0, i.e.  $\varphi(p) = d_0^2(p)$ (see e.g. \cite[Proposition 2.2]{K04}). 
Then  claim \eqref{locp} implies that  for $h \in \HH_0$ with $|h|$ small enough,
we get
\begin{align} \label{estth30}
\varphi(p \cdot h) + \varphi(p \cdot h^{-1}) - 2 \varphi(p) \le  d_0^2(p \cdot h) + d_0^2(p \cdot h^{-1}) - 2 d_0^2(p) \le C |h|^2.
\end{align}
However, recalling the horizontal differential operators introduced in  
 \eqref{firt} and \eqref{sect} and  by applying the Taylor expansion in the group (see e.g. \cite[\S~20]{BLU07}), we can write
\begin{equation} \label{estth3}
\begin{aligned}
\varphi(p \cdot h) = \varphi(p) + \langle \nabla_H \varphi(p), h \rangle + \frac{1}{2} \langle  (\nabla_H^2 \varphi(p))^* h, h \rangle + o(|h|^2), \\
\varphi(p \cdot h^{-1}) = \varphi(p) - \langle \nabla_H \varphi(p), h \rangle + \frac{1}{2} \langle  (\nabla_H^2 \varphi(p))^* h, h \rangle + o(|h|^2).
\end{aligned}
\end{equation}

Combining \eqref{estth30} and \eqref{estth3}, we obtain
\[
\langle  (\nabla_H^2 \varphi(p))^* h, h \rangle + o(|h|^2) \le C\, |h|^2,
\]
for all $h \in \HH_0$ with $|h|$ small enough. Dividing this inequality by $|h|^2$ and letting  $|h|\to $, we deduce that $(\nabla_H^2 \varphi(p))^* \le C\, \I_m$ and thus $-(\nabla_H^2 d_0^2)^* \ge -C\; \I_m$ in the viscosity sense, which concludes the result by Theorem \ref{tp2}.\\

Let us now prove {\bf claim \eqref{locp}}. Note that $\G$ is a fat, since it is assumed to be ideal.  We here only consider the case when $G$ is of step 2. The case of step 1 Carnot group can be reduced to the known Euclidean case. Let us fix the constants $C \ge 0$ and $c > 0$ appearing in Lemma \ref{l1}, and split the proof of \eqref{locp} into three cases:


\vspace{0.4cm}
 
\paragraph{Case 1: $p \in S = \{q\in \G \,|\, d(q,0) = 1\}$.} We use Lemma \ref{l1} for this case. 
It follows from from \eqref{BCDH} and  \eqref{BCDH2} that
\begin{align} \label{Lieadd}
p \cdot h = p + h + \cR(p,h), \quad p \cdot h^{-1} = p - h - \cR(p,h), \quad \forall \, p , h \in \G.
\end{align}
Furthermore, we apply inequality  \eqref{C_0}  to $p \in S = \partial B_{CC}(0,1)$ to get $|\cR(p,h)| \le C_1 |h|$, where $C_1 = C_0 \sup_{p \in S}|p| \in (0,+\infty)$. As a result, we have
$$|v_{p,h}| \le |h| + C_1 |h|, \quad\forall \, p \in S, h \in \G,\quad  \textrm{where}\; v_{p,h} := h + \cR(p,h).$$
Then for $c_1 = \frac{c}{1 + C_1}$,  whenever $p \in S$ and $h \in \HH_0$ such that $|h| \le c_1$, we deduce $|v_{p,h}| \le c$. Combining this with Lemma \ref{l1} (with constant $\overline{C}$ in place of $C$) and \eqref{Lieadd}, we obtain
\begin{align*}
&d_0^2(p \cdot h) + d_0^2(p \cdot h^{-1}) - 2 d_0^2(p) 
=  d_0^2(p + v_{p,h}) + d_0^2(p - v_{p,h}) - 2 d_0^2(p) \\
&\le  \overline{C} |v_{p,h}|^2 \le \overline{C} \left(1 + C_1\right)^2 |h|^2,
 \quad \forall \, p \in S, h \in \HH_0, |h| \le c_1.
\end{align*}
This  proves {\bf claim \eqref{locp}} with $C= \overline{C} \left(1 + C_1\right)^2 $. 


\vspace{0.4cm}

\paragraph{Case 2: $p\in \G$ and $p \ne 0$.} In this case, we use the properties of the CC distance as in \eqref{homo}.
In fact,  for every $p\neq 0$ we can define $\widetilde{p}=\delta_{1/r}(p)$ with $r=d_0(p)$ (i.e. $p=\delta_r(\widetilde{p})$) so that $\widetilde{p}\in S$. 
Indeed, we have 
\[
d_0(\widetilde{p})= d_0\left(\delta_{1/r}(p)\right)= \frac{1}{r}\, d_0(p)=\frac{1}{r}\,r=1.
\]
Then we can adopt Case 1 for $\widetilde{p}$ with $\widetilde{h}=\delta_{1/r}(h)$ for all $h\in \HH_0$ such that 
 $|h| \le c_1 r$, where $c_1$ is the constant determined in Case 1. 
It is worth pointing out that, due to the condition $h\in \HH_0$, we have
\[
\delta_{1/r}(h)=\frac{|h|}{r}, \quad |\widetilde{h}|=\frac{|h|}{r}\leq \frac{c_1\,r}{r}=c_1.
\]
Hence by our result in Case 1 we can write
\begin{align*}
&d_0^2(p \cdot h) + d_0^2(p \cdot h^{-1}) - 2 d_0^2(p)
= d_0^2(\delta_r(\widetilde{p}) \cdot h) + d_0^2(\delta_r(\widetilde{p}) \cdot h^{-1}) - 2 d_0^2(\delta_r(\widetilde{p})) \\
=\,&   r^2 [d_0^2(\widetilde{p} \cdot \delta_{1/r} (h)) + d_0^2(\widetilde{p} \cdot \delta_{1/r} (h)^{-1}) - 2 d_0^2(\widetilde{p})] 
= r^2 \left[d_0^2(\widetilde{p} \cdot \widetilde{h}) + d_0^2(\widetilde{p} \cdot \widetilde{h}^{-1}) - 2 d_0^2(\widetilde{p})\right] 
 \\
\le  \,& r^2 C |\,\widetilde{h}\,|^2 
\le\, r^2 C \frac{|h|^2}{r^2}=
 C  |h|^2,
\end{align*}
where we take $C= \overline{C} \left(1 + C_1\right)^2$ as in Case 1 and $c(p)=c_1 d_0(p)$. This proves {\bf claim \eqref{locp}} for the current case with $C= \overline{C} \left(1 + C_1\right)^2 $ and $c(p)=c_1 d_0(p)$. 

\vspace{0.4cm}

\paragraph{Case 3: $p = 0$.} It remains to prove the claim in the case $p = 0$. 
To this end, we use the equivalence of the CC distance and the homogeneous norm introduced in \eqref{homogeneousNorm}. Noticing that, for all $h \in \HH_0$ we have that $|h|_{\G}=|h|$, hence for  $p=0$ we obtain
\[
d_0^2(p \cdot h) + d_0^2(p \cdot h^{-1}) - 2 d_0^2(p)  = d_0^2(h) + d_0^2(-h) \le 2\, C_2^2\, |h|^2,
\]
 where $C_2 \ge 1$ is the constant given in Proposition \ref{phomo}. This proves {\bf claim \eqref{locp}}  for the point $p=0$ with the constants $C=2\, C_2^2$ and $c(p)=1$.

\smallskip 
 
To sum up, considering all of the cases discussed above, we have shown that {\bf claim \eqref{locp}} holds for all $p\in \G$ and $h\in \HH_0$ such that $|h|\le c(p)$, with
\[
C= \max\left(\overline{C} \left(1 + C_1\right)^2 ,2 C_2^2\right) > 0, \quad c(p) = \begin{cases}
c_1 d_0(p), \quad &\mbox{if $p \ne 0$}, \\
1, \quad &\mbox{if $p = 0$}.
\end{cases}
\]
\end{proof}

\begin{figure}[htp]
	\centering
        \begin{overpic}[scale = 0.2]{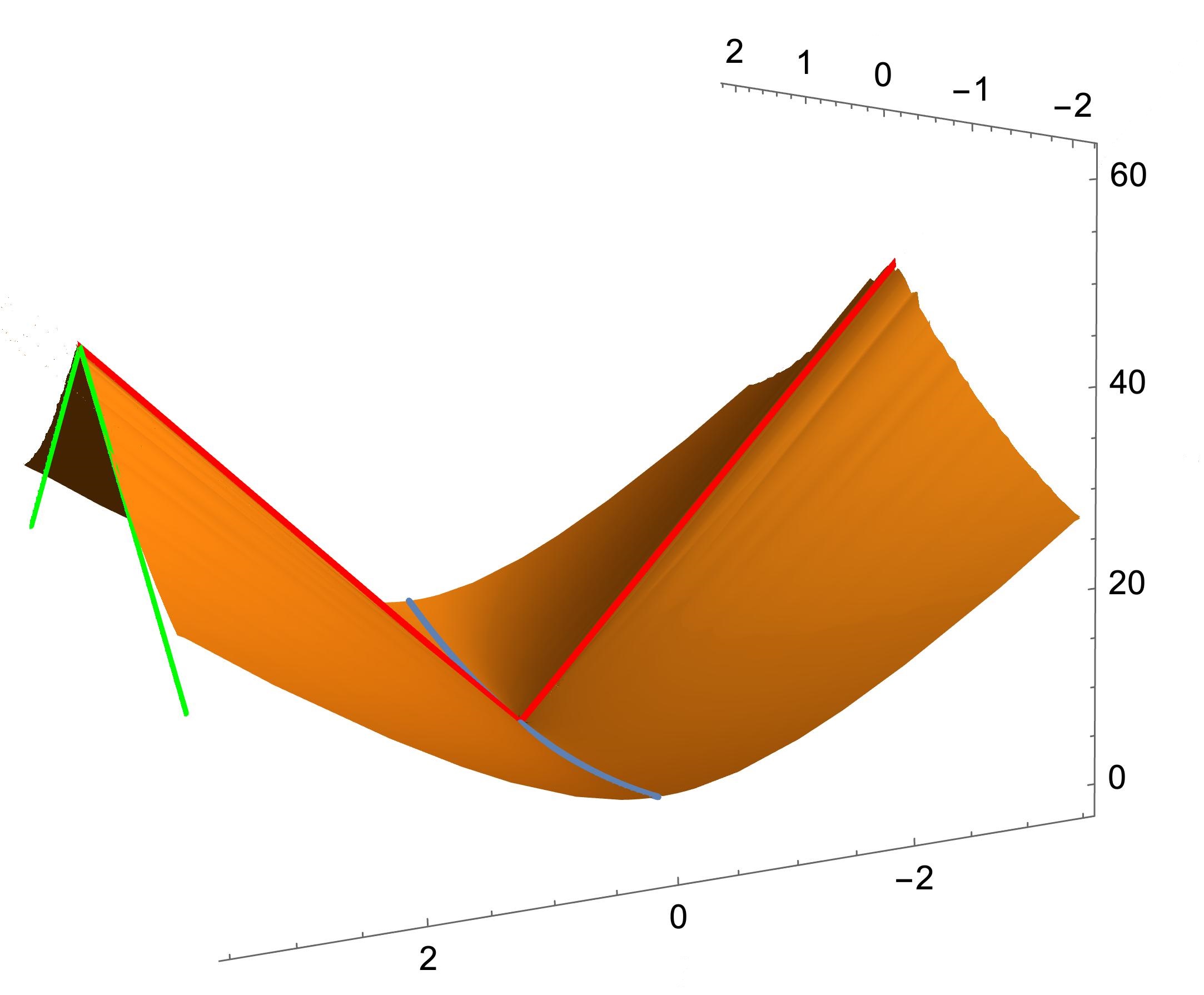}
       \put(95,42){$d_0^2$}
       \put(72,80){$x$}
       \put(55,2){$z$}
    \end{overpic}
	\caption[image]{The graph of $d_0^2=d^2(\cdot ,0)$ in the Heisenberg group $\H$}\label{fig1}
\end{figure}
\begin{remark}
We include Figure \ref{fig1} to illustrate the significant difference between the h-semiconcavity on Heisenberg group $\H$ and the usual Euclidean semiconcavity. The following explicit expression of the CC distance in $\H$ is obtained in \cite[Theorem 1.36]{BGG00}:
\begin{align}\label{expd2}
d_0^2(x,y,z) = 
\begin{dcases}
\left( \frac{\theta}{\sin{\theta}}\right)^2 (x^2 + y^2), \quad & \mbox{if $(x,y) \ne (0,0)$ and $\theta = \mu^{-1}\left(\frac{4|z|}{x^2 + y^2}\right),$} \\
4\pi |z|, \quad & \mbox{if $(x,y) = (0,0),$}
\end{dcases}
\end{align}
where $\mu: (-\pi,\pi)\to \R$  given by 
\begin{align}\label{mu}
\mu(s) := \frac{2s - \sin(2s)}{2 \sin^2{s}}
\end{align}
is an increasing diffeomorphism (cf. \cite[Lemme 3, p. 112]{G77}). Since $d_0^2$ is rotational symmetric in the coordinates $x$ and $y$, we only draw the graph of $d_0^2$ on the set $\{(x,0,z): \, x,z \in \R\}$. It can be seen from the red curve that a corner-like singularity occurs at the identity $0$ in the direction $z$ (i.e.  the forbidden direction). As a result, it is not Euclidean semiconcave at $0$, which by definition means that $d_0^2$ is not locally Euclidean semiconcave in any neighborhood of the identity. This observation might explain why the result \cite[Theorem 1]{CR08} or \cite[Theorem 5.9]{FR10} did not touch the identity on $\H$. 

Our result confirms that such singularity actually does not affect the horizontal semiconcavity of $d_0^2$ at the identity $0$. As a matter of fact, to investigate the definition of h-semiconcavity, we should restrict the function to the horizontal plane, which gives a smooth function
\[
d_0^2(x, y, 0)=  x^2 + y^2.
\]
The graph of this function is plotted as the blue curve in Figure \ref{fig1}. \end{remark}

While we have proved in Theorem \ref{t1} that the squared CC distance $d_0^2$ in an ideal Carnot group is h-semiconcave, it however fails to be an h-semiconvex function even in the Heisenberg group $\H$, the simplest example of ideal Carnot groups. To see this, we present  the following result, which suggests a corner-like singularity at every nonzero point in the center; see the green curve in Figure \ref{fig1}. See the forthcoming paper \cite{LZZ23} for more discussions about this property. Since h-semiconvexity is weaker than Euclidean semiconvexity, the following result also implies \cite[Corollary 30]{BR19} on the Heisenberg group $\H$.
\begin{proposition}
\label{cp1}
Let $d$ be the CC distance of the Heisenberg group $\H$ and $d_0=d(\cdot, 0)$ in $\H$. For every $p=(0,0,\tau)\in \H$ with $\tau\neq 0$, the following holds:
\begin{align}\label{tar1}
 \lim_{h \in \HH_0 \setminus \{0\}, h \to 0} \frac{d_0^2(p \cdot h) + d_0^2(p \cdot h^{-1}) - 2 d_0^2(p)}{|h|} = - 4 d_0(p) .
\end{align}
In particular, 
\[
 \lim_{h \in \HH_0 \setminus \{0\}, h \to 0} \frac{d_0^2(p \cdot h) + d_0^2(p \cdot h^{-1}) - 2 d_0^2(p)}{|h|^2} = - \infty.
\]

\end{proposition} 

\begin{proof}
The key to the proof is the expression of the squared CC distance given in \eqref{expd2}. Since $d_0^2$ is symmetric with respect to the $xy$-plane and rotationally symmetric about the $z$-axis, without loss of generality, we may assume that $h_1 > 0$, $h_2 = 0$, and $z > 0$.
Moreover, in view of the $1$-homogeneity of the CC distance with respect to the group dilation as shown in \eqref{homo}, it suffices to prove \eqref{tar1} for $e_3=(0,0,1)$. In fact, for a generic $p=(0,0,\tau)=\delta_\tau(e_3)$, we have, as $\HH_0\ni h\to 0$,  
 \begin{align*}
 &\frac{d_0^2(p \cdot h) + d_0^2(p \cdot h^{-1}) - 2 d_0^2(p)}{|h|}  =  \frac{\tau^2}{|h|}\left(
d_0^2\big(e_3 \cdot \delta_{1/\tau} (h)\big)  + d_0^2\big(
e_3\cdot \delta_{1/\tau} (h)^{-1}\big) - 2 d_0^2(e_3) \right)
\\
= & \frac{\tau}{|\widetilde{h}|}
\left(
d_0^2\big(e_3 \cdot\widetilde{h}\big)  + d_0^2\big(e_3\cdot  \widetilde{h}^{-1}\big) - 2 d_0^2(e_3) \right)
\to  \, -4 \tau d_0(e_3)=-4d_0(p),
\end{align*}
where we applied \eqref{tar1} at $p=e_3$ with $\widetilde{h}=\delta_{1/\tau}(h)=h/\tau$.

Let us now prove \eqref{tar1} at $p=e_3$. By symmetry, we may further take $h = (h_1,0,0)$ with $h_1 > 0$. In this case, we have $p\cdot h=(h_1,0,1)$ and  $p\cdot h^{-1}=(-h_1,0,1)$. Our goal is then to show 
\begin{align}\label{tar2}
\lim_{h_1 \to 0+} \frac{1}{h_1}\left(d_0^2(h_1,0,1) + d_0^2(-h_1,0,1)- 2 d_0^2(0,0,1)\right) = - 4d_0(0, 0, 1). 
\end{align}
We use the expression of the squared CC distance given by \eqref{expd2}, which yields
\[
d_0^2(0,0,1)= 4\pi, \qquad d_0^2(h_1,0,1) = d_0^2(-h_1,0,1) = \left( \frac{\theta}{\sin{\theta}}\right)^2 h_1^2,
\]
where $\theta = \theta(h_1) = \mu^{-1}\left( \frac{4}{h_1^2} \right) \to \pi{-}$ as $h_1\to 0+$. 
The equation for $\theta$ also gives
\begin{align}\label{asyh}
\frac{4}{h_1^2} =  \mu(\theta) =\frac{2\theta-\sin(2\theta)}{2\sin^2(\theta)},
\end{align}
which yields
\begin{equation}\label{limit1}
\frac{ h_1}{\pi-\theta}\to \frac{2}{\sqrt{\pi}},\quad \text{as $h_1 \to 0+$.}
\end{equation}
 Then using \eqref{asyh}, combined with \eqref{mu},
we get 
\begin{align*}
&\frac{d_0^2(h_1,0,1)+ d_0^2(-h_1,0,1)- 2 d_0^2(0,0,1)}{h_1} = 2h_1 \frac{d_0^2( h_1,0,1) - d_0^2(0,0,1)}{h_1^2} \\   = & \,  2 h_1 \left[ \left( \frac{\theta}{\sin{\theta}}\right)^2 - \pi \mu(\theta)\right] 
=2 h_1  \frac{\theta(\theta - \pi) + \pi \sin{\theta} \cos{\theta} }{\sin^2{\theta}} .
\end{align*}
By \eqref{limit1}, we can pass to the limit of the relation above as $h_1\to 0+$ to obtain
\[
\lim_{h_1\to 0+} \frac{d_0^2(h_1,0,1)+ d_0^2(-h_1,0,1)- 2 d_0^2(0,0,1)}{h_1}=\lim_{h_1\to 0+} -\frac{2h_1(\theta+\pi)}{\pi-\theta}=-8\sqrt{\pi}. 
\]
Our proof of \eqref{tar2} is complete, since $-8\sqrt{\pi}=-4 d_0(0,0,1)$. 
\end{proof}

\section{Generalization to step 2 Carnot groups}\label{sec:step2}
\setcounter{equation}{0}

In this section we drop the assumption that $G$ is ideal by using the results in \cite{BR18} instead of those by Cannarsa-Rifford \cite{CR08} and Figalli-Rifford \cite{FR10}. Here the main difficulty is to investigate the relation between the notion of $C$-nearly horizontally semiconcavity introduced in  \cite{BR18} and our h-semiconcavity. First we recall the  \cite[Definition 2.7]{BR18} but in a slightly simplified form. Although the definition given below is somewhat different from the original one in \cite{BR18}, it is indeed equivalent for the canonical left-invariant sub-Riemannian structure $(\D,g)$. 

In what follows we will use $D$ to denote the differential of a map between Euclidean spaces. When the target space is the real line $\R$, it is the usual gradient $\nabla$  and in such case we also use  $\nabla^2$ to denote the Hessian matrix.  Furthermore, for multi-index $\alpha = (\alpha_1 , \ldots, \alpha_m) \in \N^m$ and a vector $h \in \R^m$, we define
\begin{equation} \label{multi-index}
\alpha! := \prod_{\ell = 1}^m (\alpha_\ell)!, \quad 
|\alpha| := \sum_{\ell = 1}^m \alpha_\ell, \quad 
D^\alpha := \prod_{\ell = 1}^m \partial_\ell^{\alpha_\ell}, \quad 
h^\alpha := \prod_{\ell = 1}^m h_{\ell}^{\alpha_\ell}.
\end{equation}

\begin{definition}[$C$-nearly horizontal semiconcavity]\label{Cnhs}
Let $C > 0$ and $\Omega$ be an open subset of $\G \cong \R^n$, a function $f : \Omega \to \R$ is said to be {\it $C$-nearly horizontally semiconcave with respect to $(\D,g)$} if for every $p \in \Omega$, there
are an open neighbourhood $V_p$ of $0$ in $\HH_0 \cong \R^m$, a function $\phi_p : V_p \subset \HH_0 \to \Omega$ of class $C^2$, a function $\psi_p: V_p \subset \HH_0 \to \R$ of class $C^2$, and an $m \times m$ orthogonal matrix $O_p \in \mathrm{O}_m$ such that
\begin{align}\label{assCnh}
\phi_p(0) = p, \quad \psi_p(0) = f(p), \quad f (\phi_p (h)) \le \psi_p(h), \quad \forall \, h \in V_p, \quad 
D \phi_p (0) = D L_p(0) O_p
\end{align}
with
\[
\|\phi_p\|_{C^2} \le C, \quad \|\psi_p\|_{C^2} \le C.
\]
Here $L_p$ denotes the restriction of the left multiplication by $p$ on $\G$ on the horizontal plane $\HH_0$ and
\[
\|\varphi\|_{C^2} := \max_{\beta \in \N^m, \  |\beta| \le 2} \sup |D^\beta \varphi|.
\]
\end{definition}

Roughly speaking, the definition tells that at every point $p \in \Omega$, the function $f$ can be touching from above uniformly along some $m$-dimensional $C^2$ submanifold 
$M_p$, which is tangent to the distribution $\D_p$.
As a result, we will show that the notion of h-semiconcavity is generally stronger than this $C$-nearly horizontal semiconcavity (at least locally) since we can just choose the manifold $M_p$ to be $p \cdot \HH_0$. To show the converse, we need some additional regularity of $f$, as shown in the following proposition.

\begin{proposition}\label{propCnhs}
Let $\G$ be a step 2 Carnot group and $\Omega \subset \G$ open. Then the following results hold. 
\begin{enumerate}
\item If $f$ is h-semiconcave in $\Omega$, then for each open bounded set $\Omega'$ such that $\overline{\Omega'}\subset \Omega$, $f$ is $C$-nearly horizontally semiconcave in $\Omega'$ with some $C = C(\Omega')>0$. 
\item If $f$ is $C$-nearly horizontally semiconcave for some $C>0$ and Lipschitz with respect to the Euclidean norm, then $f$ is also h-semiconcave in $\Omega$. 
\end{enumerate}
\end{proposition}
\begin{proof}
We first prove the first assertion. In fact, since $f$ is h-semiconcave, by Theorem 
we see that the function $f_{C_0}(p) := f(p) - C_0 |p^{(1)}|^2$ is h-concave for some  constant $C_0 > 0$. It follows from the definition of 
h-concave functions that for every $p \in \Omega$, there exists a $v_p \in \R^m$ such that
\[
f_{C_0}(p \cdot h) \le f_{C_0}(p) + \langle v_p, h \rangle, \quad \forall \, h \in \HH_0,
\]
which implies
\[
f(p \cdot h) \le f(p) + \langle v_p + 2C_0 p^{(1)}, h \rangle + C_0 |h|^2, \quad \forall \, h \in \HH_0.
\]
Recalling that, for Carnot groups of step 2, we have $p \cdot h = (p^{(1)}+h, p^{(2)} + \mathcal{B}(p^{(1)},h))$ whenever  $h \in \HH_0$ it is sufficient  to choose $\phi_p(h) := L_p(h) = (p^{(1)}+h, p^{(2)} + \mathcal{B}(p^{(1)},h))$ and $\psi_p(h) := f(p) + \langle v_p + 2C_0 p^{(1)}, h \rangle + C_0 |h|^2$, to get  Definition \ref{Cnhs}.

Now we prove the second assertion. By the viscosity theory technique, namely using a similar argument in the proof of Theorem \ref{t1}, it suffices to prove the following estimate: for every $p$, there exists a constant $c(p) > 0$ such that 
\begin{align}\label{tarp}
f(p \cdot h) + f(p \cdot h^{-1}) - 2 f(p) \le 2(L + 1) C m^2 |h|^2, \qquad \forall \, h \in \HH_0, |h| \le c(p),
\end{align}
where $C$ is the constant in the definition of $C$-nearly horizontally semiconcavity, and $L> 0$ is the Lipschitz constant of  $f$, i.e.
\begin{align}\label{Lipf}
|f(\b{p}) - f(\t{p})| \le L |\b{p} - \t{p}|, \qquad \forall \, \b{p}, \t{p} \in \Omega.
\end{align}

We first need to estimate the ``distance'' between $M_p$ and $p \cdot \HH_0$. By Taylor expansion, for every $\wt{h} \in \HH_0$ with $|\,\wt{h} \, |$ small enough, we have $\phi_p(\wt{h}) = p  + D \phi_p(0) \wt{h} + r_{p,\wt{h}}$, where the remainder is
\[
r_{p,\wt{h}} = \sum_{\beta \in \N^m, \ |\beta| = 2} \frac{2}{\beta !} \left( \int_0^1 (1 - t) D^\beta \phi_p(t \wt{h}) dt \right) (\wt{h})^\beta.
\]
Here, we recall from \eqref{multi-index} the notations related to the multi-index $\beta$. 
From the assumption $\|\phi_p\|_{C^2} \le C$, we have 
\[
\left|r_{p,\wt{h}}\right| \le \sum_{\beta \in \N^m,\ |\beta| = 2} 2C |\,\wt{h} \,|^2 \int_0^1 (1 - t) dt  \le C m^2|\,\wt{h} \,|^2,
\]
and, choosing $\b{p} = p + D \phi_p(0) \wt{h}$ 
and $\t{p} = \phi_p(\wt{h})$  in \eqref{Lipf}, we obtain
\begin{align}\label{Lipf2}
f(p + D \phi_p(0) \wt{h}) - f(\phi_p(\wt{h})) \le |f(p + D \phi_p(0) \wt{h}) - f(\phi_p(\wt{h}))| \le L \left|r_{p,\wt{h}}\right| \le L C  m^2|\,\wt{h} \,|^2.
\end{align}

Recalling that $L_p(h) = p \cdot h = (p^{(1)} + h, p^{(2)} + \mathcal{B}(p^{(1)},h)), \forall \, p \in \G, h \in \HH_0$ with $\mathcal{B}$ bilinear, we have
\[
p + D L_p(0) h = (p^{(1)} , p^{(2)}) + (h, \mathcal{B}(p^{(1)},h)) = p \cdot h.
\]
Plugging the previous equation into \eqref{Lipf} with $h = O_p \wt{h}$ and using \eqref{assCnh}, we deduce
\begin{align}\label{Lipf3}
f(p \cdot h) = f(p + D L_p(0) h) = f(p + D L_p(0) O_p \wt{h}) = f(p + D \phi_p(0) \wt{h}) \le f(\phi_p(\wt{h})) +  L C  m^2|\,\wt{h} \,|^2.
\end{align}
Similarly, by Taylor expansion we deduce
\begin{align}\label{Tay1}
\psi_p(\wt{h}) \le   f(p) + \langle D \psi_p(0), \wt{h} \rangle + C m^2 | \, \wt{h} \, |^2 = f(p) +  \langle O_p D \psi_p(0), O_p \wt{h} \rangle + C m^2 |\,\wt{h} \,|^2,
\end{align}
where we used the fact that $O_p$ is orthonormal. 
Plugging \eqref{Lipf3} and \eqref{Tay1} into \eqref{assCnh} and for $h: = O_p \wt{h}$ (note that by orthonormality $|h|=|\wt{h}|$), we obtain
\begin{equation}
\label{cesare}
f(p \cdot h) \le f(p) + \langle O_p D \psi_p(0), h \rangle + (L + 1) C m^2 |\,h \,|^2, \qquad \forall \, h \in \HH_0 \mbox{ with $| \, h \, |$ small enough}.
\end{equation}

Writing \eqref{cesare}   for $h$ and  for $h^{-1}=-h$, 
and adding two inequalities together, we can deduce, for all $h\in \HH_0$ small enough
\[
f(p \cdot h) + f(p \cdot h^{-1}) - 2f(p) \le  2(L + 1) C m^2 |h|^2,
\]
which implies h-semiconcavity in view of the viscosity interpretation as in the proof of Theorem~\ref{t1}. 
\end{proof}




The assumption of Euclidean Lipschitz continuity is essential in order to control all the terms coming from the second order Taylor expansion in the group and in particular the first derivatives in the forbidden directions appearing there. In fact, while h-semiconcavity implies local Lipschitz continuity w.r.t. to the CC distance that allows to bound the horizontal gradient  (cf. \cite{DGN03, M06, M062, SY06}), in general the total gradient may be unbounded. It turns out that such an assumption is necessary for the equivalence to hold true, as one can see in the following example.


\begin{example}\label{ce}
On the Heisenberg group $\H$ in Example \ref{e1}, the CC distance from the origin $d_0$ is not h-semiconcave but $C$-nearly horizontally semiconcave on $\Omega_{a,b,R} := \{(x,y,z) \in \H \,|\, a (x^2 + y^2) < z < b (x^2 + y^2), |z| < R\}$ for some suitable $a,b,R,C > 0$. Now we choose these constants. We refer the reader to Appendix \ref{sec:Ab} for the detailed calculation. Considering a point $p_* = (x_*,y_*,z_*)$ with $z_* > 0$ and $(x_*,y_*) \ne (0,0)$ and letting $z_* \to 0+$ with $(x_*,y_*)$ fixed, the auxiliary function $\theta = \theta(p_*) \to 0+$ since $\mu^{-1}(s) \to 0+$ as $s \to 0+$. Then it follows from \eqref{pd2d2}-\eqref{pd2d4} and \eqref{HeHe}-\eqref{HeHe2} that
\[
(\nabla_H^2 [d_0^2](p_*))^* \to 2 \, \I_2 + \frac{4}{\mu'(0)} \frac{1}{x_*^2 + y_*^2} 
\begin{pmatrix}
y_*^2 & - x_* y_* \\
- x_* y_* & x_*^2
\end{pmatrix} > 0
\]
since $\mu(s) \to 0+$ as $s \to 0+$. As a result, for $z_*$ small enough there exists a point $p_*$ such that  $ (\nabla_H^2 [d_0^2](p_*))^*$ is positive definite. Now fix the point $p_*$ and we choose $a,b,R$ in such a way that the $p_* \in \Omega_{a,b,R}$. By direct computation, we have 
\[
 (\nabla_H^2 [d_0](p_*))^* = \frac{1}{2 d_0(p_*)} [(\nabla_H^2 [d_0^2](p_*))^* - 2 (\nabla_H [d_0](p_*))\otimes ( \nabla_H [d_0](p_*))].
\]
Note that $(\nabla_H^2 [d_0](p_*))^*$ has at least positive eigenvalue, since if we pick a nonzero vector $\nu$ orthogonal to $\nabla_H [d_0](p_*)$, we have $\langle (\nabla_H^2 [d_0](p_*))^* \nu, \nu \rangle > 0$. Consequently,  $(\nabla_H^2 [d_0])^*$ cannot be bounded from above in $\Omega_{a,b,R}$. In fact, by definition we have $\delta_r (p_*) \in \Omega_{a,b,R}$ for every $r \in (0,1)$, and the dilation property \eqref{homo} yields
\[
(\nabla_H^2 [d_0](\delta_r (p_*)))^* = \frac{1}{r} (\nabla_H^2 [d_0](p_*))^*.
\]
This shows that $d_0$ is not h-semiconcave in $\Omega_{a,b,R}$. 

Let us now prove that in $\Omega_{a,b,R}$, $d_0$ is $C$-nearly horizontally semiconcave for some $C > 0$. For every $p_0 = (x_0,y_0,z_0) \in \Omega_{a,b,R}$, we define the following two functions 
\[
\begin{aligned}
    &\phi_{p_0}(x,y) := \left(x + x_0, y + y_0, z_0 + \frac{x_0 y - y_0 x}{2} - C (x^2 + y^2)\right),\\
    &\quad \psi_{p_0}(x,y) := d_0(p_0) + \langle \nabla_H [d_0](p_0), (x,y) \rangle. 
\end{aligned}
\]
From definition it is clear that 
\[
\phi_{p_0}(0) = p_0, \quad \psi_{p_0}(0) = d_0(p_0),  \quad 
D \phi_{p_0} (0) = D L_{p_0}(0)
\]
and the $C^2$ norms are uniformly bounded (if the parameter $C$ is given).
If we can show $\nabla^2 [d_0 \circ \phi_{p_0}](0) < 0$, then in a small neighbourhood of $0$ we have $d_0(\phi_{p_0}(x,y)) \le d_0(p_0) + \langle \nabla_H [d_0](p_0), (x,y) \rangle$ since $\nabla [d_0 \circ \phi_{p_0}](0) = \nabla_H [d_0](p_0)$. Noticing that 
\[
\nabla^2 [d_0 \circ \phi_{p_0}](0) = \frac{1}{2 d_0(p_0)} [\nabla^2 [d_0^2 \circ \phi_{p_0}](0)- 2 (\nabla_H [d_0](p_0))(\nabla_H [d_0](p_0))^\T]
\]
we only need to prove $\nabla^2 [d_0^2 \circ \phi_{p_0}](0) < 0$. From \eqref{D2t} we have
\[
\nabla^2 [d_0^2 \circ \phi_{p_0}](0) = (\nabla_H [d_0^2](p_0))^* +  (\partial_z d_0^2) \, \nabla^2 \phi_{p_0}(0) = (\nabla_H [d_0^2](p_0))^* - 8 \theta(p_0) C  \, \I_2.
\]
Since $(\nabla_H [d_0^2](p_0))^*$ is bounded from above and 
\[
0 < \mu^{-1}(4a) < \theta(p_0) < \mu^{-1}(4b), \qquad \forall \, p_0 \in \Omega_{a,b,R},
\]
we can always choose a constant $C > 0$ such that $\nabla^2 [d_0^2 \circ \phi_{p_0}](0) < 0$. We have verified that $d_0$ satisfies Definition \ref{Cnhs}.
\end{example}

Now we prove Theorem \ref{t2} for general step 2 Carnot groups. To this end, we first establish the following lemma. Since $T_0^* \G \cong (T_0 \G)^* \cong \g^* \cong \R^n$, there is a natural norm $|\cdot|$ on $T_0^* \G$.

\begin{lemma}\label{lp29}
Let $\G$ be a step 2 Carnot group with CC distance $d$.  Given a fixed (nonempty) compact set $K \subset \G$, there exists a constant $A > 0$ such that for every $p \in K$ and every length minimizing geodesic $\gamma : [0,1] \to \G$ from $0$ to $p$, there is $\xi \in T_0^* \G$ with $|\xi| \le A$ such that $\gamma$ is the projection of some normal extremal with the initial covector $\xi$. 
\end{lemma}

\begin{proof}
Without loss of generality we can assume $K$ is the CC ball $\overline{B_{CC}(0,M)}$, otherwise just enlarge the set $K$ since it is bounded.

First, note that each length minimizing geodesic in a step 2 Carnot group is normal (cf. \cite[Theorem 2.22]{R14}). In particular, it is the projection of some normal extremal and thus real analytic. 

Second, for any $p \in \overline{B_{CC}(0,M)} \setminus \mathrm{Cut}_0$, where $\mathrm{Cut}_0$ denotes the cut locus of the identity $0$ defined by
\[
\mathrm{Cut}_0 := \{p \in \G\,| \,  d_0^2 \mbox{ is not smooth in a neighborhood of $p$}\},
\]
we know (\cite[Lemma 2.15]{R14}, \cite[Proposition 2]{RT05})  that there exists a unique length minimizing geodesic $\gamma : [0,1] \to \G$ from $0$ to $p$. It is thus a projection of some normal extremal $\mu: [0,1] \to T^* \G$ with $\mu(1) = D[d_0^2](p)/2 \in T_p^* \G$. Here we recall that $D$ denotes the differential. In fact, in the setting of step 2 Carnot groups, after choosing suitable coordinates, we can identify $T^* \G$ with $\R^n \times \R^n$ and write down the Hamilton equation \eqref{sRH} explicitly. See \cite[\S~13.1]{ABB20} for more details.  Then it follows from \cite[(13.9)]{ABB20} that $|\mu(1)| = |\mu(0)|$. Hence, to prove $|\mu(0)|$ is uniformly bounded, it is sufficient to show that $|\mu(1)|$ is uniformly bounded. In fact, this boundedness comes from the local Lipschitz continuity of $d_0^2$ with respect to Euclidean norm (cf. \cite[Corollary 12.14]{ABB20}).

Concerning points $p\in \overline{B_{CC}(0,M)} \cap \mathrm{Cut}_0$, since the cut locus has measure zero (\cite[Proposition 15]{R13}), we can always approximate $p$ by points in $\overline{B_{CC}(0,M)} \setminus \mathrm{Cut}_0$. Since $d_0^2$ is locally Lipschitz, up to a subsequence, the differentials also converge. It is guaranteed by \cite[Proposition 4]{RT05} that the limiting length minimizing geodesic has the same bound as before.

In particular, what we have proved so far implies that if the length minimizing geodesic from $0$ to $p$ is unique, then we can find such a $A$ in the assertion. 

Finally we show that the bound $2A$ works for all remaining points.  Let $\gamma$ be an arbitrary length minimizing geodesic. By real analyticity, length minimizing geodesics cannot branch (cf. \cite[Proposition 10]{MR20}) and thus for every $s \in (0,1)$, the restriction $\gamma|_{[0,s]}$ (aftering reparameterization to $[0,1]$) is a unique length minimizing geodesic from $0$ to $\gamma(s)$. In particular, we choose $s = 1/2$ and by the result above, $\gamma|_{[0,1/2]}$ is the projection of some normal extremal with the initial covector $\xi$ with $|\xi| \le A$. From \cite[Lemma 8.35]{ABB20}, the projection of the normal extremal with the initial covector $2\xi$ is the same as $\gamma$ on $[0,1/2]$. By real analyticity again they coincide on the whole $[0,1]$. This ends the proof with the new bound $2A$.
\end{proof}

Let us now prove Theorem \ref{t2}, using the $C$-nearly horizontal semiconcavity result in \cite{BR18}. 


\begin{proof}[Proof of Theorem \ref{t2}]
Similar to the proof of Theorem \ref{t1}, we need only to prove the local estimates corresponding to the ones in Lemma \ref{l1}. 
It follows from \cite[Proposition 2.9]{BR18} and Lemma \ref{lp29} that $d_0^2$ on step 2 Carnot group $\G$ is $C$-nearly horizontally semiconcave and locally  Lipschitz w.r.t. the Euclidean norm in $B_{CC}(0,2)$ (see \cite[Corollary 12.14]{ABB20}), thus h-semiconcave in $B_{CC}(0,2)$ by Proposition \ref{propCnhs}. Then for any $p \in S =  \partial B_{CC}(0,1)$ and $h \in \HH_0$ with $|h| < 1/2$, we have $[p \cdot h^{-1}, p \cdot h] \subset B_{CC}(0,2) $ since
\[
|d_0(p \cdot th) - d_0(p)| \le d(p \cdot th,p) = d_0(th) = t|h| < 1, \qquad \forall \, t \in [-1,1].
\]
From the h-semiconcavity, there exists a $C_\#$ such that
\[
d_0^2(p \cdot h) + d_0^2(p \cdot h^{-1}) - 2 d_0^2(p) \le C_\# |h|^2, \qquad \forall \, p \in S, h \in \HH_0, |h| \le 1/2.
\]
This gives the desired estimate and completes the proof of Theorem \ref{t2}.
\end{proof}


The assumption of step 2 in Theorem \ref{t2} is essential. The h-semiconcavity of the square of the CC distance actually fails to hold on the Engel group introduced in Example~\ref{e3}. See \cite[\S~4]{MM16} for details about this observation. 
More precisely, the following result holds. 


\begin{proposition}\label{cp2}
On the Engel group $\E$, there exists a nonzero element $p$ in $\E$ (in the abnormal set) such that the following limit holds:
\begin{equation}\label{cp2 eq1}
 \lim_{h \in \R \setminus \{0\},\ h \to 0} \frac{d_0(p \cdot h e_1) - d_0(p)}{h^2} =+\infty,
\end{equation}
where $e_1 := (1,0,0,0) \in \HH_0 \subset \E$. In particular, for every $r > 0$, the following holds:
\begin{equation}\label{cp2 eq2}
 \lim_{h \to 0+} \frac{d_0^r(p\cdot h e_1) + d_0^r(p \cdot (h e_1)^{-1}) - 2 d_0^r(p)}{h^2} =+\infty.
\end{equation}
\end{proposition} 

We refer to \cite[Theorem 1.2]{MM16} for the proof of \eqref{cp2 eq1}. As an immediate consequence, one can obtain \eqref{cp2 eq2} by the mean value theorem.




\section{Further consequences and generalization}\label{sec:rc}
\setcounter{equation}{0}

Applying Theorem \ref{tp2} to the function $d_0^2=d^2(\cdot, 0)$, a direct consequence of Theorem \ref{t2} is the following corollary.

\begin{corollary} \label{c1}
Let $d$ be the CC distance of a step 2 Carnot group $\G$. For $d_0=d(\cdot, 0)$,   there exists a constant $C > 0$ such that 
$$- (\nabla^2_H [d_0^2])^* \ge -C \, \I_m 
\ \  
 \textrm{in $\G$ holds in the viscosity sense.}
$$ 
In particular, $- \Delta_H [d_0^2] \geq -mC$ in $\G$ holds in the viscosity sense, where $m$ denotes the dimension of the first layer of the Lie algebra of $\G$. Here, the viscosity inequalities mean that $- (\nabla^2_H \varphi(p))^* \ge -C \, \I_m$ and $-\Delta_H \varphi(p) \ge -mC$  hold for any  $\varphi \in C^2(\G)$ and $p \in \G$ such that $d_0^2 - \varphi$ attains a local minimum at $p$. 
\end{corollary}

\begin{proof}
From Theorem \ref{t2} we know that $d_0^2$ is h-semiconcave. Then by (ii) of Theorem \ref{tp2}, there exists a constant $C \ge 0$ such that for any $\varphi \in C^2(\G)$ and $p \in \G$ with the property that $d_0^2 - \varphi$ has a local minimum at $p$, we have $-(\nabla_H^2 \varphi(p))^* \geq - C \, \I_m$. Taking the trace on both sides, we obtain $- \Delta_H \varphi(p) + m C \ge 0$, which concludes the proof. 
\end{proof}

\begin{corollary} \label{c2}
Let $d$ be the CC distance of a step 2 Carnot group $\G$. For $d_0=d(\cdot, 0)$, there exists a constant $C > 0$ such that 
\[
(\nabla^2_H [d_0^2] )^*\le C \, \I_m,
\]
and 
\[
\Delta_H [d_0^2] \le mC
\]
hold almost everywhere in $\G$. 
\end{corollary}

\begin{proof}
We first define the {\it cut locus (of the identity $0$)} as follows:
\[
\mathrm{Cut}_0 := \{p \in \G\,| \,  d_0^2 \mbox{ is not smooth in a neighbourhood of $p$}\}.
\]
It is known \cite[Proposition 15]{R13} that the cut locus has measure zero if the step of a Carnot group is two. 
From Corollary \ref{c1} we know that $ - (\nabla^2_H [d_0^2])^* \ge -C \, \I_m$ and 
$- \Delta_H [d_0^2] \ge -mC$ hold in the viscosity sense. Therefore, by standard techniques of viscosity solution theory,   $ - (\nabla^2_H [d_0^2](p))^* \ge -C \, \I_m$ and $- \Delta_H [d_0^2](p) \ge -mC$ hold at all points $p$ where $d_0^2$ is smooth including all $p \in \G\setminus \mathrm{Cut}_0$. Then the conclusion of the corollary follows.
\end{proof}

It is possible to generalize our result in Theorem \ref{t2} for other functions related to the CC distance. We present the following generalization in a bounded open set of a step 2 Carnot group. 

\begin{corollary} \label{compof}
Let $\G$ be a step 2 Carnot group with CC distance $d$. Let $\Omega$ be a bounded open set of $\G$. Assume that $\Psi:[0,+\infty)\to [0,+\infty)$ is an increasing function such that its even extension $\widetilde{\Psi}: \R \to [0,+\infty)$ is a locally semiconcave function in $\R$, as defined in \eqref{lsemiconc}. Then, $\Psi(d(\cdot, 0))$ is h-semiconcave in $\Omega$.

\end{corollary}
\begin{proof}
We still use the notation $d_0=d(\cdot, 0)$ in $\G$.
We first assume that $\widetilde{\Psi} \in C^2(\R)$. Since $\widetilde{\Psi}$ is even, it is easy to obtain $\widetilde{\Psi}'(0) = 0$. Now let $T(\Omega) := \sup_{p \in \Omega} d(p) < +\infty$. On the compact set $[0,T(\Omega)]$, by the local semiconcavity there exists a constant $C(\Omega ) > 0$ such that $\widetilde{\Psi}'' \le C(\Omega)$. Consequently this implies
\begin{align}\label{estdp}
0 \le \widetilde{\Psi}'(\tau) = \widetilde{\Psi}'(\tau) - \widetilde{\Psi}'(0) \le C(\Omega) \tau, \qquad \forall \, \tau \in [0,T(\Omega)].
\end{align}
Moreover, for $p \in \Omega \setminus \{0\}$ and $h \in \HH_0$ such that $[p \cdot h^{-1}, p \cdot h] \subset \Omega$, by Taylor expansion we have
\begin{align}\label{estdiff}
&\Psi(d_0(p \cdot h)) - \Psi(d_0(p)) \le  \widetilde{\Psi}'(d_0(p))(d_0(p \cdot h) - d_0(p)) + \frac{C(\Omega )}{2}(d_0(p \cdot h) - d_0(p))^2.
\end{align}

Notice that
\[
\widetilde{\Psi}'(d_0(p))(d_0(p \cdot h) - d_0(p))
\le  \frac{\widetilde{\Psi}'(d_0(p))}{2 d_0(p)}  (d_0^2(p \cdot h) - d_0^2(p)),
\]
and, also by the fact $h \in \HH_0$, we have
\[
|d_0(p \cdot h) - d_0(p)| \le d(p \cdot h,p) = d_0(h) = |h|,
\]
since it is easy to see that $t \to th, t \in [0,1]$ is a length minimizing geodesic. Inserting these two estimates into \eqref{estdiff}, we obtain
\[
\Psi(d_0(p \cdot h)) - \Psi(d_0(p)) \le  \frac{\widetilde{\Psi}'(d_0(p))}{2 d_0(p)} (d_0^2(p \cdot h) - d_0^2(p)) + \frac{C(\Omega )}{2} |h|^2.
\]
Similarly we have 
\[
\Psi(d_0(p \cdot h^{-1})) - \Psi(d_0(p)) \le  \frac{\widetilde{\Psi}'(d_0(p))}{2 d_0(p)} (d_0^2(p \cdot h^{-1}) - d_0^2(p)) + \frac{C(\Omega )}{2} |h|^2.
\]
Adding them together and apply Theorem \ref{t2} as well as \eqref{estdp},  we have
\begin{align*}
&\Psi(d_0(p \cdot h)) + \Psi(d_0(p \cdot h^{-1})) - 2\Psi(d_0(p)) \\
\le \, & \frac{\widetilde{\Psi}'(d_0(p))}{2 d_0(p)} (d_0^2(p \cdot h) + d_0^2(p \cdot h^{-1}) - 2 d_0^2(p)) + C(\Omega ) |h|^2 \\
\le \, & \frac{\widetilde{\Psi}'(d_0(p))}{2 d_0(p)} C|h|^2  + C(\Omega ) |h|^2 \le (C/2 + 1) C(\Omega ) |h|^2,
\end{align*}
under the assumption that $p \in \Omega \setminus \{0\}$ and $h \in \HH_0$ such that $[p \cdot h^{-1}, p \cdot h] \subset \Omega$, where $C$ is the h-semiconcavity constant of $d_0^2$. This estimate still holds for $p = 0$, since $\widetilde{\Psi}'(d_0(p)) = \widetilde{\Psi}'(0) = 0$ in \eqref{estdiff}. This ends the proof for the case $\widetilde{\Psi} \in C^2(\R)$. 

For general $\widetilde{\Psi}$, it is sufficient to approximate by standard mollification. Take $\phi_\varepsilon(\cdot) = \varepsilon^{-1} \phi(\cdot /\varepsilon)$, where $\phi$ is an even, nonnegative, smooth function with compact support such that $\int_{\R} \phi \,d x = 1$ and decreasing on $[0,+\infty)$. For every $\varepsilon \in (0,1)$, we can show that $\phi_\varepsilon \ast \widetilde{\Psi}$ is smooth, even, increasing on $[0,+\infty)$, and locally semiconcave with the semiconcavity constant on any compact set independent of $
\varepsilon \in (0,1)$. Since $\phi_\varepsilon \ast \widetilde{\Psi}\to \widetilde{\Psi}$ locally uniformly as $\varepsilon\to 0$, we can apply the standard stability argument for viscosity solutions to conclude our proof. 
\end{proof}

\begin{remark}
Typical examples of the function $\Psi$ satisfying the assumptions of Corollary \ref{compof} are $\Psi(\tau)=C \tau^\gamma$ with $C > 0$ and $\gamma \ge 2$.
    \end{remark}


We conclude this section by remarking that the strategy employed here can be generalized to cover step 2 sub-Riemannian manifolds, such as the rototraslation geometry introduced by Citti and Sarti \cite{cittiSarti} to model the visual cortex. The main difference is that, due to the absence of dilations, the result would only hold locally, i.e., within compact sets. In this context, the notion of h-semiconcavity needs to be adapted to deal with the lack of an algebraic structure, using the definition of $\cX$-semiconcavity introduced in \cite{BD11}. Although the definition in \cite{BD11} is given for H\"ormander vector field structures in $\R^n$, the same notion can be extended to general manifolds in local coordinates, as it remains invariant under changes of coordinates. By similar arguments, we can prove the following generalization of Theorem \ref{t2}.

\begin{theorem}
Given a manifold $M$ equipped with a step 2 analytic sub-Riemannian  structure $(\D, g)$. Then for every $x \in M$, there exists a neighborhood $U_x$ such that 
\[
\D = \mathrm{span} \{\X_1, \ldots, \X_m\}
\]
and $d_x^2$ is $\mathcal{X}$-semiconcave on $U_x$. Here $\mathcal{X} := \{\X_1, \ldots, \X_m\}$.
\end{theorem}

\section{Application to Hamilton-Jacobi equations}\label{sec:app}
\setcounter{equation}{0}

In this section, we study h-semiconcavity of the viscosity solutions for a class of time-dependent Hamilton-Jacobi equations of the form:
\begin{equation}
\label{H-J_general}
\left\{
\begin{aligned}
&u_t+\Phi\big(|\nabla_H u|)=0,\quad \mbox{in} \quad (0,+\infty) \times \G,\\
&u(0,\cdot)=g, \quad \mbox{on} \quad \{0\} \times \G,
\end{aligned}
\right.
\end{equation}
where $u_t$ denotes the time derivative of $u$, $\nabla_H u$ is the horizontal gradient in a (step 2) Carnot group $\G$, and $\Phi: [0,+\infty) \to [0,+\infty)$ is a continuous, convex, non-decreasing function such that $\Phi(0) = 0$.
We assume throughout this section that $g\in LSC(\G)$, where we recall that $LSC(A)$ denotes the set of lower semicontinuous functions in a set $A$. In \cite{D07}  (see also  \cite{MS02} for the case of the Heisenberg group)
it was proved that, if $g\in LSC(\G)$ and 
\begin{equation}
\label{DatumSuperlinear}
\exists C>0 \quad \text{such that} \quad g(p) \ge -C (1 + d_0(p)) \quad \text{holds for all $p \in \G$}, \
\end{equation}
then the viscosity solution $u\in LSC([0, +\infty)\times\G)$ of the Cauchy problem \eqref{H-J_general} can be obtained by the (metric) Hopf-Lax formula
\begin{align}\label{HLF}
u(t,p) = \inf_{q \in \G} \left[  g(q) +t\Phi^*\left( \frac{d(p,q)}{t}\right)\right], \quad (t, p)\in [0, +\infty)\times \G, 
\end{align}
where $\Phi^*$ is the Legendre-Fenchel function associated to $\Phi$, that is defined by 
\begin{equation}\label{legendre}
\Phi^*(s):=\sup_{\tau\geq 0}\big\{ s\tau-\Phi(\tau)\}, \quad s\geq 0.
\end{equation}
For the uniqueness, we refer to \cite{Cutri}, where comparison principles are proved for the Cauchy problem with continuous initial data.
Using Theorem \ref{t2}, we prove that, under suitable conditions on $
\Phi^*$, for all $t>0$, the viscosity solution of problem \eqref{H-J_general} given by \eqref{HLF} is h-semiconcave in space. For the convenience of the reader, we will first show the result in the easiest case when $\Phi(s)=s^2/2$ for $s\geq 0$ and then study the case of a more general $\Phi$. 

Let us first recall another known result for the Hopf-Lax function.

\begin{lemma}[\cite{D07}]\label{D07}
Let $\G$ be a step 2 Carnot group with CC distance $d$. Let $d_0=d(\cdot, 0)$ in $\G$. Assume that $g\in LSC(\G)$ and satisfies \eqref{DatumSuperlinear}.  Then $u\in LSC([0, +\infty)\times\G)$ and there exists a constant $C'>0$ such that
\begin{align}\label{lowu}
u(t,p) \ge - C' (1 + d_0(p) + t), \quad \forall \, p \in \G, t > 0.
\end{align}
Moreover, if $g\in LSC(\G)$ is bounded, then the infimum in \eqref{H-J_general} is actually a minimum and it is attained in a CC ball centred at the point $p$ with radius depending only on $\Phi$ and $t$.
\end{lemma}

Our first result is the following.

\begin{theorem}
\label{PropSquareH-L}
Let $\G$ be a step 2 Carnot group with CC distance $d$. Let $d_0=d(\cdot, 0)$ in $\G$. Assume that $g\in LSC(\G)$ satisfies \eqref{DatumSuperlinear}. Let $\Phi(s)=s^2/2$ for $s\geq 0$ and $u$ be defined as in \eqref{HLF}. Then $u(t,\cdot)$ is h-semiconcave in $\G$, for every $t > 0$.
\end{theorem}

\begin{proof}
Since $d(p,q)=d(q^{-1}\cdot p,0) = d_0(q^{-1}\cdot p)$ holds for all $p, q\in \G$, given the choice $\Phi(s)=s^2/2$ for $s\geq 0$, we have $\Phi^*(s)=s^2/2$ by \eqref{legendre} and the function $u$ in \eqref{HLF} reduces to 
$$
u(t,p) = \inf_{q \in \G} \left\{  g(q) + \frac{d_0^2(q^{-1} \cdot p)}{2t}\right\}, \quad (t, p)\in [0, +\infty)\times \G.  
$$
By Theorem \ref{t2}, for every $q \in \G$ and $t > 0$, the function 
\[
p \mapsto g(q) + \frac{d_0^2(q^{-1} \cdot p)}{2t}
\]
is h-semiconcave with h-semiconcavity constant $C/(2t)$,  where $C>0$ is the h-semiconcavity constant of $d_0^2$. In view of \eqref{lowu}, we have $u(t,p) > -\infty$ for any $t>0$ and $p\in \G$. Then the h-semiconcavity of $u(t, \cdot)$ in $\G$ follows from Proposition \ref{pinf}. \end{proof}
\begin{remark} 
\label{Rem_Min}
Theorem \ref{PropSquareH-L} implies that  $u(t,\cdot)$ defined by \eqref{HLF} is also locally Lipschitz continuous with respect to the CC distance, which in turn yields a local H\"older continuity with respect to the Euclidean distance. 
\end{remark}
We next generalize the previous result for more general Hamilton-Jacobi equations but under additional boundedness assumption on $g$ and locally strong convexity of the Hamiltonian. Here we say that a function $f$  is {\it locally strongly convex}  in an open set $\Omega\subset \R^n$ if for every compact convex set $K \subset \R$, there exists a constant $C(K) > 0$ such that 
\[
\lambda f(x) + (1 - \lambda) f(y) - f(\lambda x + (1 - \lambda) y) \ge \lambda (1 - \lambda)C(K) |x - y|^2, \ \  \forall \, x , y \in K, \lambda \in [0,1]. 
\]

\begin{theorem}
\label{TheoremH-J}
Let $\G$ be a step 2 Carnot group with CC distance $d$. Assume that $g\in LSC(\G)$ is bounded. Let $\Phi:[0,+\infty)\to [0,+\infty)$ be a continuous, convex, non-decreasing function with $\Phi(0) = 0$. Assume in addition that $\Phi$ is coercive in the sense that 
\begin{equation}\label{coercive}
\frac{\Phi(\tau)}{\tau}\to +\infty\quad \text{as $\tau\to +\infty$,}
\end{equation}
and the even extension $\widetilde{\Phi}: \R\to [0, +\infty)$ of $\Phi$ is locally strongly convex in $\R$. 
 Let $u$ be the viscosity solution of \eqref{H-J_general} defined as in \eqref{HLF}.
 Then, $u(t, \cdot)$ is h-semiconcave in $\G$ for every $t > 0$.
\end{theorem}
\begin{proof}
It is well known that the Legendre-Fenchel transform of a coercive, strongly convex function  is semiconcave in $\R$; see for example \cite[Lemma 4 in Chapter 3.4]{E10}. Using \eqref{coercive}, one can localize this property to prove the Legendre-Fenchel transform of a coercive, locally strongly convex function is locally semiconcave in $\R$. Since $\widetilde{\Phi}$ is locally strongly convex in $\R$, we thus obtain the local semiconcavity of $(\widetilde{\Phi})^\ast$ in $\R$. 

On the other hand, as $\Psi=\Phi^\ast$ in $[0, +\infty)$, we have $(\widetilde{\Phi})^\ast\big|_{[0, +\infty)}=\Psi$ under current assumptions on $\Phi$; in other words, $\Psi$ has a semiconcave even extension in $\R$. Now applying Corollary \ref{compof} with $\Psi=\Phi^\ast$, we obtain the local h-semiconcavity of $\Phi^\ast(d_0)$, that is $\Phi^\ast(d_0)$ is h-semiconcave in any bounded open set $\Omega\subset \G$. 

Let us fix $t>0$ arbitrarily. For any $p\in \G$, we can use the boundedness of $g$ to deduce that 
\[
g(q)+t \Phi^\ast \left(\frac{d_0(q^{-1}\cdot p)}{t}\right)\to +\infty\quad \text{as $d_0(q)\to +\infty$. }
\]
By \eqref{HLF}, it then follows that there exists $\hat{q}\in \G$ depending on $p$ such that 
\[
u(t, p)=g(\hat{q})+t \Phi^\ast \left(\frac{d_0(\hat{q}^{-1}\cdot p)}{t}\right)\leq g(p), 
\]
which yields 
\[
t \Phi^\ast \left(\frac{d_0(\hat{q}^{-1}\cdot p)}{t}\right)\leq g(p)-g(\hat{q}).
\]
Applying the boundedness of $g$ again, we are led to $d_0(\hat{q}^{-1}\cdot p)< M$ for some $M>0$ depending on $\Phi, t$ but independent of $p, \hat{q}$. In other words, for every fixed $t>0$ and $p\in \G$, we have 
\begin{equation}\label{HLF-bound2}
u(t, p)=\inf_{\Omega}\left\{g(q)+t \Phi^\ast \left(\frac{d_0({q}^{-1}\cdot p)}{t}\right)\right\},
\end{equation}
for any open set $\Omega\subset \G$ satisfying $B_{CC}(p, M)\subset \Omega$. 

By the local h-semiconcavity of $\Phi^\ast(d_0)$, for any $p_0\in \G$, we see that  
\[
p\mapsto g(q)+t \Phi^\ast \left(\frac{d_0({q}^{-1}\cdot p)}{t}\right)
\]
is h-semiconcave in $B_{CC}(p_0, 1)$ for all $q\in B_{CC}(p_0, M+1)$ with h-semiconcavity constant depending only on $M$ (in particular independent of $p_0$). In view of Proposition \ref{pinf}, we obtain the h-semiconcavity of $u(t, \cdot)$ in $B_{CC}(p_0, 1)$ with the same h-semiconcavity constant by taking infimum of the function above over $q\in B_{CC}(p_0, M+1)$ and noting that
\[
u(t, \cdot)=\inf_{q\in B_{CC}(p_0, M+1)} \left\{g(q)+t \Phi^\ast \left(\frac{d_0({q}^{-1}\cdot p)}{t}\right)\right\},
\] 
thanks to \eqref{HLF-bound2} with $\Omega=B_{CC}(p_0, M+1)\supset B_{CC}(p, M)$.

Since the h-semiconcavity constant of $u(t, \cdot)$ in $B_{CC}(p_0, 1)$ is independent of $p_0\in \G$, we thus obtain the h-semiconcavity of $u(t, \cdot)$ in $\G$. 
\end{proof}

Our h-semiconcavity result above can be applied to some particular Hamilton-Jacobi equations such as 
\[
u_t+\frac{1}{\alpha} |\nabla_H u|^\alpha=0 \quad \text{in $(0, +\infty)\times \G$},
\]
with a bounded initial value $g\in LSC(\G)$ and $1<\alpha\leq 2$. The case $\alpha=1$ is not covered by Theorem \ref{TheoremH-J}, but if $g\in LSC(\G)$ is assumed to be bounded and h-semiconcave in $\G$, then we have preservation of the spatial h-semiconcavity of the viscosity solution given by the optimal control formula
\[
u(t, p)=\inf_{q\in B_{CC}(p, t)} g(q),\quad t>0,\ p\in \G.
\]
The proof is simply a straightforward application of Proposition \ref{pinf}. 

It is not our intention to study in detail stationary PDE problems in this paper, but one possible simple application of Theorem \ref{t2} in this direction is for the  eikonal equation 
\begin{equation}\label{eikonal}
|\nabla_H u|=1 \quad \text{in $\Omega$},    
\end{equation}
where $\Omega\subset \G$ is a given open set. Let $\bS =\G\setminus \Omega$, we define the {\it CC distance from the set $\bS$} by
\[
d_\bS(p) := \min_{q \in \bS} d(p,q) = \min_{q \in \bS} d_0(q^{-1} \cdot p).
\]
Note that $d_\bS$ is continuous, and when $\bS = \{0\}$, $d_\bS=d_0$, which is exactly the CC distance from the group identity. It is well known \cite{D07} that $u=d_\bS$ is a viscosity solution of \eqref{eikonal} satisfying the boundary condition $u=0$ on $\partial \Omega$; we refer to \cite{Bi10} for more discussions about the distance function and eikonal equation. We can use Theorem \ref{t2} to prove easily that the square of the solution $d_\bS$ is h-semiconcave in $\Omega$. 

\begin{proposition}\label{pd2}
Let $\bS \subset \G$ be a nonempty closed set on a step 2 Carnot group $\G$. Then $d_\bS^2$ is an h-semiconcave function in $\G\setminus \bS$.
\end{proposition}

\begin{proof}
Observe that 
\[
d^2_\bS(p) = \min_{q \in \bS} d^2(p, q)= \min_{q \in \bS} d_0^2(q^{-1} \cdot p),
\]
and for every $q \in \bS$, the function $p \mapsto d_0^2(q^{-1} \cdot p)$ is h-semiconcave with h-semiconcavity constant the same as the one of $d_0^2$. As a result, the proof follows from Proposition \ref{pinf}. 
\end{proof}

\appendix

\section{Properties of special functions and details of computation}\label{sec:Ab}
\setcounter{equation}{0}

This appendix provides additional details for the arguments presented in Example \ref{ce}.
Recall that from definition \eqref{mu} we have
\[
\mu(s) = -(s \cot s)' = \frac{s - \sin{s} \cos{s}}{\sin^2{s}} = \frac{2s - \sin(2s)}{2 \sin^2{s}}.
\]
See Figure \ref{fig4} for the graph of $\mu$. 
\begin{figure}[htp]
	\centering
  \begin{overpic}[scale = 0.4]{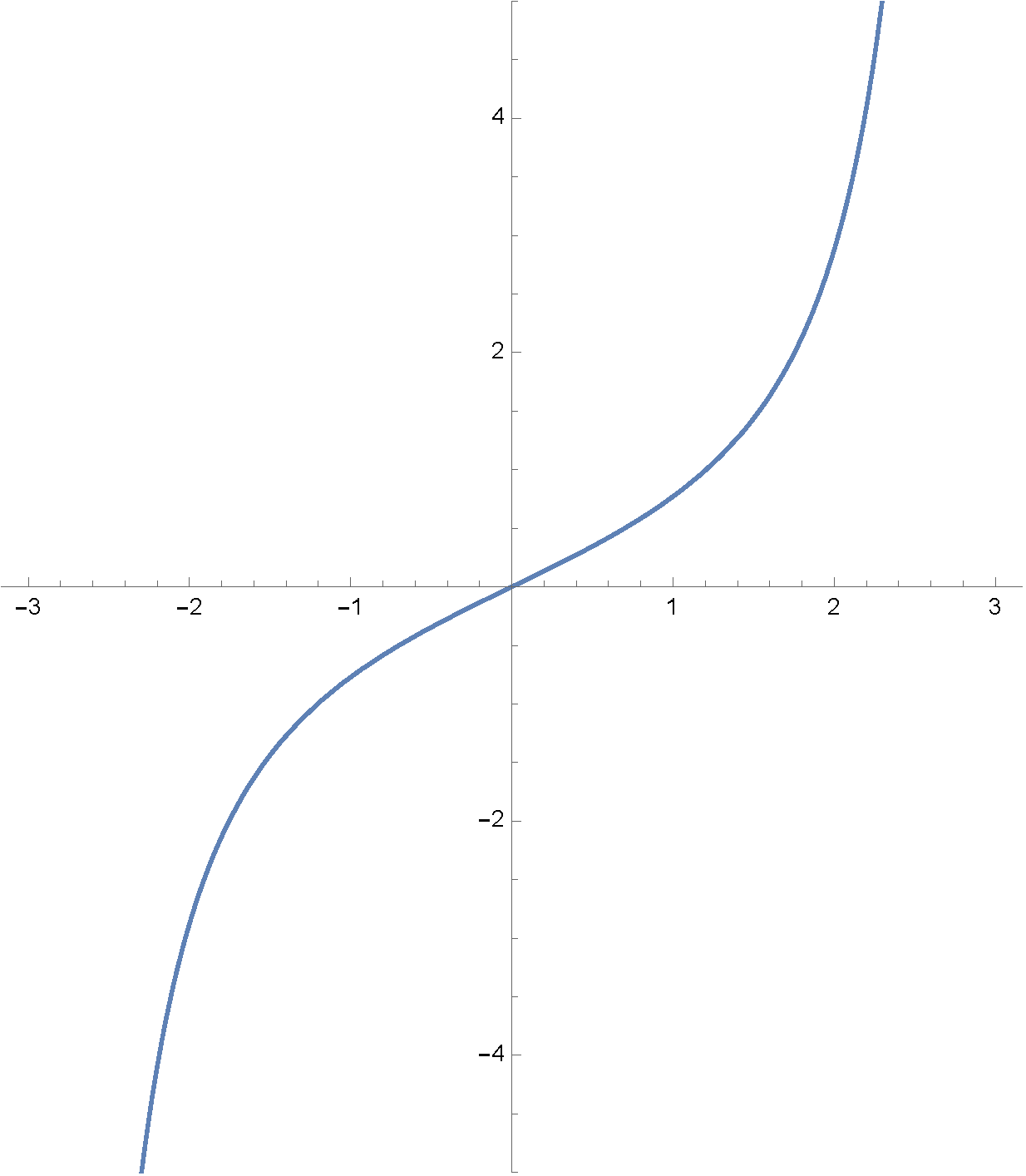}
       \put(60,70){$\mu$}
   \end{overpic}
	\caption[image]{The graph of the function $\mu$}\label{fig4}
\end{figure}

By direct computations, we have
\begin{align}\label{dmu}
\mu'(s) = \frac{2 (\sin{s} - s \cos{s})}{\sin^3{s}} > 0, \qquad \forall \, s \in (-\pi, \pi),
\end{align}
which implies $\mu$ is  an increasing diffeomorphism between $(-\pi, \pi)$ and $\R$. Consult also \cite[Lemme~3, p.~112]{G77}. As a result, the inverse function $\mu^{-1}$ is an increasing diffeomorphism between $\R$ and $(-\pi, \pi)$. By the inverse function theorem we have
\begin{align}\label{dmui}
(\mu^{-1})'(s) = \frac{1}{\mu'(\mu^{-1}(s))}, \qquad \forall \, s \in \R.
\end{align}

Recall that we have the formula of $d_0^2$ \eqref{expd2}. For $z>0$ and $(x,y) \neq (0,0)$, the auxiliary function $\theta = \theta(x,y,z) = \mu^{-1}\left(\frac{4z}{x^2 + y^2}\right)$ is a smooth function of $(x,y,z)$. As a result, by \eqref{dmui}, we can calculate the derivatives of $\theta$:
\begin{align}\label{dtx}
\partial_x \theta = (\mu^{-1})'\left(\frac{4z}{x^2 + y^2}\right)  \left(-\frac{4z}{(x^2 + y^2)^2}\right)  2x = -\frac{2x \mu(\theta)}{(x^2 + y^2) \mu'(\theta)}
\end{align}
and similarly
\begin{align}\label{dtyz}
\partial_y \theta =  -\frac{2y \mu(\theta)}{(x^2 + y^2) \mu'(\theta)}, \qquad 
\partial_z \theta = \frac{4}{(x^2 + y^2) \mu'(\theta)}.
\end{align}

Now using \eqref{dtx} we obtain
\begin{align}\nonumber
&\partial_x d_0^2 = 2 \left( \frac{\theta}{\sin{\theta}}\right) \frac{\sin{\theta} - \theta \cos{\theta}}{\sin^2{\theta}} \partial_x \theta (x^2 + y^2) + \left( \frac{\theta}{\sin{\theta}}\right)^2 2x \\
\nonumber
= \,& 2 \left( \frac{\theta}{\sin{\theta}}\right) \frac{\sin{\theta} - \theta \cos{\theta}}{\sin^2{\theta}} \left(-\frac{2x \mu(\theta)}{(x^2 + y^2) \mu'(\theta)} \right) (x^2 + y^2)  +  \left( \frac{\theta}{\sin{\theta}}\right)^2 2x \\
\label{dd2x}
= \, & 2x \left( \frac{\theta}{\sin{\theta}}\right) \left(2 \frac{\sin{\theta} - \theta \cos{\theta}}{\sin^2{\theta}}\frac{-\theta + \sin{\theta \cos{\theta}}}{\sin^2{\theta}} \frac{\sin^3{\theta}}{2 (\sin{\theta} - \theta \cos{\theta})} + \frac{\theta}{\sin{\theta}} \right) = 2 \theta \cot{\theta} \, x.
\end{align}
For a similar reason, 
\begin{align}\label{dd2yz}
\partial_y d_0^2 = 2 \theta \cot{\theta} \, y, \qquad
\partial_z d_0^2 = 4 \theta.
\end{align}

Taking derivatives again we have
\begin{align}\label{pd2d2}
\partial_{xx} d_0^2 = 2 \theta \cot{\theta} + \frac{4x^2 \mu^2(\theta)}{(x^2 + y^2) \mu'(\theta)}, \qquad
&\partial_{xy} d_0^2 = \frac{4xy \mu^2(\theta)}{(x^2 + y^2) \mu'(\theta)}, \\
\label{pd2d3}
\partial_{yy} d_0^2 = 2 \theta \cot{\theta} + \frac{4y^2 \mu^2(\theta)}{(x^2 + y^2) \mu'(\theta)}, \qquad 
&\partial_{xz} d_0^2 = 4 \partial_x \theta = -\frac{8x \mu(\theta)}{(x^2 + y^2) \mu'(\theta)}, \\
\label{pd2d4}
\partial_{yz} d_0^2 = 4 \partial_y \theta = -\frac{8y \mu(\theta)}{(x^2 + y^2) \mu'(\theta)}, \qquad 
&\partial_{zz} d_0^2 = 4 \partial_z \theta =  \frac{16}{(x^2 + y^2) \mu'(\theta)}.
\end{align}


Fixing a $p_0 = (x_0,y_0,z_0) \in \Omega_{a,b,R}$ in Example \ref{ce}, from the definition of $\phi_{p_0}$, we have
\[
(d_0^2 \circ \phi_{p_0})(x,y) =  d_0^2\left(x + x_0, y + y_0, z_0 + \frac{x_0 y - y_0 x}{2} - C (x^2 + y^2)\right).
\]

For the ease of notation we use $F_{p_0}$ to denote $d_0^2 \circ \phi_{p_0}$. Then a direct computation gives
\begin{align*}
\partial_x F &= (\partial_x d_0^2) \circ \phi_{p_0} +  \left( - \frac{y_0}{2} - 2 C x \right) (\partial_z d_0^2) \circ \phi_{p_0}, \\
\partial_y F &= (\partial_y d_0^2) \circ \phi_{p_0} +  \left(  \frac{x_0}{2} - 2 C y \right) (\partial_z d_0^2) \circ \phi_{p_0}, 
\end{align*}
and 
\begin{align*}
\partial_{xx} F &= (\partial_{xx} d_0^2) \circ \phi_{p_0}  +  2  \left( - \frac{y_0}{2} - 2 C x \right) (\partial_{xz} d_0^2) \circ \phi_{p_0} \\
& \hspace{1.2cm} + \left( - \frac{y_0}{2} - 2 C x \right)^2 (\partial_{zz} d_0^2)\circ \phi_{p_0} - 2 C (\partial_z d_0^2) \circ \phi_{p_0}, \\
\partial_{xy} F &= (\partial_{xy} d_0^2) \circ \phi_{p_0}  +  \left(  \frac{x_0}{2} - 2 C y \right) (\partial_{xz} d_0^2) \circ \phi_{p_0} +  \left( - \frac{y_0}{2} - 2 C x \right) (\partial_{yz} d_0^2) \circ \phi_{p_0} \\
 & \hspace{1.2cm} + \left(  \frac{x_0}{2} - 2 C y \right) \left( - \frac{y_0}{2} - 2 C x \right) (\partial_{zz} d_0^2)\circ \phi_{p_0}, \\
\partial_{yy} F &= (\partial_{yy} d_0^2) \circ \phi_{p_0}  +  2  \left(  \frac{x_0}{2} - 2 C y \right) (\partial_{yz} d_0^2) \circ \phi_{p_0}\\
& \hspace{1.2cm} + \left(  \frac{x_0}{2} - 2 C y \right)^2 (\partial_{zz} d_0^2)\circ \phi_{p_0} - 2 C (\partial_z d_0^2) \circ \phi_{p_0}.
\end{align*}
Inserting the point $0$, by \eqref{HeHe}-\eqref{HeHe2} and \eqref{dd2yz}, we obtain
\begin{align}\label{D2t}
\nabla^2 [d_0^2 \circ \phi_{p_0}](0) &= \nabla^2 F_{p_0}(0) = (\nabla_H [d_0^2](p_0))^* - 8 \theta(p_0) C  \, \I_2.
\end{align}

\end{document}